\documentclass[11pt,reqno]{amsart}
\usepackage{amsmath,amssymb,latexsym,soul,cite,mathrsfs}
\usepackage{color,enumitem,graphicx}
\usepackage[colorlinks=true,urlcolor=blue,
citecolor=red,linkcolor=blue,linktocpage,pdfpagelabels,
bookmarksnumbered,bookmarksopen]{hyperref}
\usepackage[english]{babel}
\usepackage{enumitem}

\usepackage[left=2.7cm,right=2.7cm,top=2.9cm,bottom=2.9cm]{geometry}

\makeatletter
\newcommand{\leqnomode}{\tagsleft@true}
\newcommand{\reqnomode}{\tagsleft@false}
\makeatother

\numberwithin{equation}{section}

\newtheorem{theorem}{Theorem}[section]
\newtheorem{lemma}[theorem]{Lemma}

\newtheorem{proposition}[theorem]{Proposition}

\newtheorem{remark}[theorem]{Remark}

\renewcommand{\rightarrow}{\to}

\providecommand{\ln}{\mathop{\rm ln}\nolimits}

\newenvironment{acknowledgement}{\noindent\textbf{Acknowledgements}\em}{}

\title[Ground states for a class of $(p,q)$-Laplacian coupled systems]{Positive ground states for a class of superlinear $(p,q)$-Laplacian coupled systems involving Schr\"{o}dinger equations}

\author[J.C. \ de Albuquerque]{J. C. de Albuquerque}
\author[J.M.\ do \'O]{Jo\~ao Marcos do \'O}
\author[Edcarlos D. Silva]{Edcarlos D. Silva}

\address[J.C. de~Albuquerque]{Department of Mathematics, Federal University of Goi\'{a}s}
\email{\href{mailto:joserre@gmail.com}{joserre@gmail.com}}

\address[J.M. do \'O]{Department of Mathematics,
	Federal University of Para\'{\i}ba
	\newline\indent
	58051-900, Jo\~ao Pessoa-PB, Brazil}
\email{\href{mailto:jmbo@pq.cnpq.br}{jmbo@pq.cnpq.br}}

\address[E.D. Silva]{Department of Mathematics,
	Federal University of Goi\'{a}s
	\newline\indent
	74001-970, Goi\'{a}s-GO, Brazil}
\email{\href{mailto:eddomingos@hotmail.com}{eddomingos@hotmail.com}}

\thanks{Corresponding author: jmbo@pq.cnpq.br}

\thanks{Research supported in part by INCTmat/MCT/Brazil, CNPq and CAPES/Brazil. The third author was also partially supported by Fapego Fapeg/CNpq grants 03/2015-PPP}



\begin{document}
	

\begin{abstract}
	We study the existence of positive ground state solutions for the following class of  $(p,q)$-Laplacian coupled systems
	\[
	\left\{
	\begin{array}{lr}
	-\Delta_{p} u+a(x)|u|^{p-2}u=f(u)+ \alpha\lambda(x)|u|^{\alpha-2}u|v|^{\beta}, & x\in\mathbb{R}^{N},\\
	-\Delta_{q} v+b(x)|v|^{q-2}v=g(v)+ \beta\lambda(x)|v|^{\beta-2}v|u|^{\alpha}, & x\in\mathbb{R}^{N},
	\end{array}
	\right.
	\]
	where $N\geq3$ and $1< p\leq q<N$. Here the coefficient $\lambda(x)$ of the coupling term is related with the potentials by the condition $|\lambda(x)|\leq\delta a(x)^{\alpha/p}b(x)^{\beta/q}$ where $\delta\in(0,1)$ and $\alpha/p+\beta/q=1$. We deal with periodic and asymptotically periodic potentials. The nonlinear terms $f(s), \; g(s)$ are ``superlinear" at $0$ and at $\infty$ and are assumed without the well known Ambrosetti-Rabinowitz condition at infinity. Thus, we have established the existence of positive ground states solutions for a large class of nonlinear terms and potentials. Our approach is variational and based on minimization technique over the Nehari manifold.
	
\end{abstract}

\maketitle




\section{Introduction} 
\label{intro}

In this work we study the following class of $(p,q)$-Laplacian coupled systems
\begin{equation}\label{ej0}
\left\{
\begin{array}{lr}
-\Delta_{p} u+a(x)|u|^{p-2}u=f(u)+\alpha\lambda(x)|u|^{\alpha-2}u|v|^{\beta}, & x\in\mathbb{R}^{N},\\
-\Delta_{q} v+b(x)|v|^{q-2}v=g(v)+ \beta\lambda(x)|v|^{\beta-2}v|u|^{\alpha}, & x\in\mathbb{R}^{N},
\end{array}
\right.
\end{equation}
where $ 1<p\leq q<N$ and $N \geq 3$. We are concerned with the existence of \textit{ground state solutions}, that is, solutions with minimal energy among the energy of all nontrivial solutions. We study a general class of $(p,q)$-Laplacian coupled systems, when the potentials $a(x),b(x)$ are nonnegative, bounded and related with the coupling term by the following estimate $|\lambda(x)|\leq\delta a(x)^{\alpha/p}b(x)^{\beta/q}$, for some $\delta\in(0,1)$ and for all $x\in\mathbb{R}^{N}$ with $\alpha/p+\beta/q=1$ and $1 \leq \alpha < p,\; 1 \leq \beta < q $. Notice that this class of systems is a type of
``$(p,q)$-linearly coupled system" due the presence of the powers $\alpha$ and $\beta$ in the coupling terms. Another feature of this class of systems is the loss of homogeneity due the fact that we consider also the case $p \neq q$. We consider the case when these functions are periodic and asymptotically periodic, that is, the limits of $a(x),b(x)$ and $\lambda(x)$ are periodic functions when $|x|\rightarrow+\infty$ in a suitable sense. Latter on, we shall discuss the assumptions on the potentials $a(x), b(x)$ and $\lambda(x)$. The nonlinearities $f(s)$ and $g(s)$ are two continuous $(p,q)$-superlinear and subcritical functions which do not satisfy the Ambrosetti-–Rabinowitz condition at infinity. In fact, we suppose that $f(s)$ is $p$-superlinear and $g(s)$ is $q$-superlinear. Our main contribution here is to prove the existence of positive ground state solutions for a general class of $(p,q)$-coupled systems defined in the whole space $\mathbb{R}^{N}$ which include several particular classes of nonlinear Schr\"{o}dinger equations and linearly coupled systems.

\subsection{Motivation and Related Results }

In order to introduce the study of the class of $(p,q)$-Laplacian coupled systems \eqref{ej0}, we begin by giving a survey on the related problems which motivates the present work. When $\lambda=0$, $f\equiv g$, $a=b$ and $p=q$, System~\eqref{ej0} reduces to the following class of quasilinear Schr\"{o}dinger equations
\begin{equation}\label{ej2}
-\Delta_{p}u+a(x)u=f(u), \quad x\in\mathbb{R}^{N}.
\end{equation}
Equations involving the $p$-Laplacian operator arise in various branches of mathematical physics, such as non-Newtonian fluids, elastic mechanics, reaction-diffusion problems, flow through porous media, glaciology, petroleum extraction, nonlinear optics, plasma physics, nonlinear elasticity, etc. We refer to \cite{mawhin} and \cite{peral} for more details about the $p$-Laplacian and \cite{fisica22} for informations about applications involving this operator. When $p=2$, solutions of \eqref{ej2} are related with standing wave solutions of the nonlinear Schr\"{o}dinger equation
\begin{equation}\label{ej5}
i\hbar\frac{\partial\psi}{\partial t}=-\frac{\hbar^{2}}{2m}\Delta\psi+\tilde{a}(x)\psi-f(\psi), \quad x\in\mathbb{R}^{N}, \ t\geq0,
\end{equation}
where $i$ denotes the imaginary unit and $m,\hbar$ are positive constants. For \eqref{ej5}, a solution of the form $\psi(x,t)=e^{-\frac{iEt}{\hbar}}u(x)$ is called \textit{standing wave}. Assuming that $f(t\xi)=f(t)\xi$ for $\xi\in\mathbb{C}$, $|\xi|=1$, taking $\hbar=2m$ and denoting $a(x)=\tilde{a}(x)-E$, it is well known that $\psi$ is a solution of \eqref{ej5} if and only if $u$ solves equation \eqref{ej2}. For more information on the physical background, we refer the readers to \cite{fisica1,fisica2,fisica22,fisica3} and references therein.

The class of equations \eqref{ej2} has been extensively studied by many researchers. In order to overcome the difficulty originated from the lack of compactness, the authors introduced several classes of potentials. For instance, in \cite{rabinowitz}, P. Rabinowitz studied Schr\"{o}dinger equations when the potential is coercive and bounded away from zero. In order to improve the behavior of the potential introduced in \cite{rabinowitz}, T. Bartsch and Z.Q. Wang, \cite{bw}, considered a class of potentials such that the level sets $\{x\in\mathbb{R}^{N}:a(x)\leq M\}$ have finite Lebesgue measure for all $M>0$. Here we deal with two classes of nonnegative bounded potentials. For more results concerning nonlinear Schr\"{o}dinger equations we refer the readers to \cite{perera,alves,liu,cm,liu2,mana} and references therein.

Our goal in this paper is to prove the existence of positive ground state solutions for the general class of coupled systems \eqref{ej0}. In order to establish a variational approach to our problem, throughout all the paper we assume that
\begin{equation}\label{ej1}
\frac{\alpha}{p}+\frac{\beta}{q}=1 \quad \mbox{and} \quad \left\{\begin{array}{rrl}p<\alpha+\beta<q, & \mbox{if} & p<q,\\
\alpha+\beta=p=q, & \mbox{if} & p=q.
\end{array}
\right.
\end{equation}
The prototypical example when $p=q=2$ and $\alpha=\beta=1$ is the following linearly coupled system
\begin{equation}\label{ej11}
\left\{
\begin{array}{lr}
-\Delta u+a(x)u=f(u)+\lambda(x)v, & x\in\mathbb{R}^{N},\\
-\Delta v+b(x)v=g(v)+\lambda(x)u, & x\in\mathbb{R}^{N}.
\end{array}
\right.
\end{equation}
In \cite{jr1,jr2}, the authors studied the existence of positive ground states for \eqref{ej11} when $N=2$. For the case $N\geq2$ we refer the readers to \cite{ambrocolo,acr,czo1,czo2,dc,maia,maia2} and references therein. In \cite{vel3}, J. V\'{e}lin studied the existence of solutions for the following $(p,q)$-gradient elliptic system with boundary Dirichlet conditions
\[
\left\{
\begin{array}{cl}
-\Delta_{p}u =\gamma a(x)|u|^{p-2}u+f(x,u,v), & x\in\Omega,\\
-\Delta_{q}v =\delta b(x)|v|^{q-2}v+g(x,u,v), & x\in\Omega,\\
u=v=0, & x\in\partial\Omega.
\end{array}
\right.
\]
In \cite{li}, C. Li and C-L. Tang proved the existence of at least three weak solutions to the following class of quasilinear elliptic systems
\[
\left\{
\begin{array}{cl}
-\Delta_{p}u =\lambda F_{u}(x,u,v), & x\in\Omega,\\
-\Delta_{q}v =\lambda F_{v}(x,u,v), & x\in\Omega,\\
u=v=0, & x\in\partial\Omega.
\end{array}
\right.
\]
For more existence results concerning to $(p,q)$-Laplacian elliptic systems we refer the readers to \cite{boc,vel1,vel2,liuliu,grego} and references therein. We point out that in the most of these works, it was considered problems defined in bounded domains and it was obtained the existence of solution.

Motivated by the above discussion, we study the class of $(p,q)$-Laplacian coupled systems \eqref{ej0}. As we mentioned in \eqref{ej1}, we study the $(p,q)$-Laplacian system \eqref{ej0} when $p=q$ or $p\neq q$. This class of systems imposes some difficulties. The first one is the lack of compactness due to the fact that the system is defined in the whole Euclidean space $\mathbb{R}^{N}$. Moreover, System~\eqref{ej0} involves strongly coupled Schr\"{o}dinger equations because of the coupling terms in the right hand side. Another difficulty is that the nonlinear terms does not verify the well known Ambrosetti-Rabinowitz condition. Namely, it says that: There exists $\theta>2$ such that
\begin{equation}\label{ar}
0 < \theta F(t) =\theta\int_{0}^{t} f(\tau)\;\mathrm{d} \tau\leq t f(t), \quad \mbox{for all} \hspace{0,2cm} t \in \mathbb{R}. \tag{AR}
\end{equation}
The Ambrosetti-Rabinowitz condition plays an important role in studying the existence of solutions to elliptic equations of variational type. For instance, it is usually used to guarantee the boundedness of Palais-Smale or Cerami sequences. Instead \eqref{ar}, we suppose that $f$ is $p$-superlinear and $g$ is $q$-superlinear. In order to obtain ground states, we use a variational approach based on minimization technique over the Nehari manifold.

\subsection{Assumptions and main result}

Firstly, we are interested in to establish the existence of positive ground state solutions for the following class of linearly coupled systems involving quasilinear Schr\"{o}dinger equations
\begin{equation}\label{paper2ej0}
\left\{
\begin{array}{lr}
-\Delta_{p} u+a_{\mathrm{o}}(x)|u|^{p-2}u=f(u)+ \alpha\lambda_{\mathrm{o}}(x)|u|^{\alpha-2}u|v|^{\beta}, & x\in\mathbb{R}^{N},\\
-\Delta_{q} v+b_{\mathrm{o}}(x)|v|^{q-2}v=g(v)+ \beta\lambda_{\mathrm{o}}(x)|v|^{\beta-2}v|u|^{\alpha}, & x\in\mathbb{R}^{N},
\end{array}
\right. \tag{$S_{\mathrm{o}}$}
\end{equation}
where $N\geq3$, $1<p\leq q<N$ and $a_{\mathrm{o}}(x),b_{\mathrm{o}}(x),\lambda_{\mathrm{o}}(x)$ are periodic potentials.

For $s>1$, let $W^{1,s}(\mathbb{R}^{N})$ be the usual Sobolev space with the norm
\[
\|u\|_{W^{1,s}(\mathbb{R}^{N})}=\left(\int_{\mathbb{R}^{N}}|\nabla u|^{s}\;\mathrm{d} x+\int_{\mathbb{R}^{N}}|u|^{s}\;\mathrm{d} x\right)^{1/s}.
\]
In view of the presence of the potential $a_{\mathrm{o}}(x)$, we introduce the following space and norm
\[
E_{a_{\mathrm{o}},p}=\left\{u\in W^{1,p}(\mathbb{R}^{N}):\int_{\mathbb{R}^{N}}a_{\mathrm{o}}(x)|u|^{p}\;\mathrm{d} x<+\infty\right\}, \quad \|u\|_{a_{\mathrm{o}},p}^{p}=\int_{\mathbb{R}^{N}}(|\nabla u|^{p}+a_{\mathrm{o}}(x)|u|^{p})\;\mathrm{d} x.
\]
Analogously, in view of the presence of the potential $b_{\mathrm{o}}(x)$, we introduce
\[
E_{b_{\mathrm{o}},q}=\left\{v\in W^{1,q}(\mathbb{R}^{N}):\int_{\mathbb{R}^{N}}b_{\mathrm{o}}(x)|v|^{q}\;\mathrm{d} x<+\infty\right\}, \quad \|v\|_{b_{\mathrm{o}},q}^{q}=\int_{\mathbb{R}^{N}}(|\nabla v|^{q}+b_{\mathrm{o}}(x)|u|^{q})\;\mathrm{d} x.
\]
We set the product space $E_{\mathrm{o}}=E_{a_{\mathrm{o}},p}\times E_{b_{\mathrm{o}},q}$ which is a reflexive Banach space when endowed with the norm
$
\|(u,v)\|_{\mathrm{o}}=\|u\|_{a_{\mathrm{o}},p}+\|v\|_{b_{\mathrm{o}},q}.
$
In order to establish a variational approach to treat System~\eqref{paper2ej0}, we need to require suitable assumptions on the potentials. Throughout the paper, we assume that:
\begin{enumerate}[label=($V_{1}$),ref=$(V_{1})$]
	\item \label{v1}
	$a_{\mathrm{o}},b_{\mathrm{o}},\lambda_{\mathrm{o}}\in C(\mathbb{R}^{N})$ are $1$-periodic in each $x_{1},x_{2},...,x_{N}$.
\end{enumerate}
\begin{enumerate}[label=($V_{2}$),ref=$(V_{2})$]
	\item \label{v3}
	$a_{\mathrm{o}}(x),b_{\mathrm{o}}(x)\geq0$, for all $x\in\mathbb{R}^{N}$ and
	\[
	\inf_{u\in E_{a_{\mathrm{o}},p}}\left\{\int_{\mathbb{R}^{N}}|\nabla u|^{p}\;\mathrm{d} x+\int_{\mathbb{R}^{N}}a_{\mathrm{o}}(x)|u|^{p}\;\mathrm{d} x:\int_{\mathbb{R}^{N}}|u|^{p}\;\mathrm{d} x=1\right\}>0,
	\]
	\[
	\inf_{v\in E_{b_{\mathrm{o}},q}}\left\{\int_{\mathbb{R}^{N}}|\nabla v|^{q}\;\mathrm{d} x+\int_{\mathbb{R}^{N}}b_{\mathrm{o}}(x)|v|^{q}\;\mathrm{d} x:\int_{\mathbb{R}^{N}}|v|^{q}\;\mathrm{d} x=1\right\}>0.
	\]
\end{enumerate}
\begin{enumerate}[label=($V_{3}$),ref=$(V_{3})$]
	\item \label{v2}
	We assume $|\lambda_{\mathrm{o}}(x)|\leq\delta a_{\mathrm{o}}(x)^{\alpha/p}b_{\mathrm{o}}(x)^{\beta/q}$, for some $\delta\in(0,1)$ such that
	\[
	\frac{1}{q}-\delta\max\left\{\frac{\alpha}{p},\frac{\beta}{q}\right\}>0.
	\]
\end{enumerate}
\begin{enumerate}[label=($V_{3}'$),ref=$(V_{3}')$]
	\item \label{v3'}
	We suppose \ref{v2} holds and there exists $R>0$ such that $\lambda_{\mathrm{o}}(x)\geq\lambda_{0}>0$, for all $x\in B_{R}(0)$.
\end{enumerate}

\noindent In this work the main interest is to ensure existence of ground state by minimization on the Nehari manifold. For this purpose we assume that $\sup_{t > 0} f'(t)t/f(t)<+\infty, \sup_{t > 0} g'(t)t/g(t)<+\infty$. Furthermore, we make the following assumptions on the nonlinearities:
\begin{enumerate}[label=($F_1$),ref=$(F_1)$]
	\item \label{f1}
	$f,g\in C^{1}(\mathbb{R})$, $f(t)=o(|t|^{p-2}t)$, $g(t)=o(|t|^{q-2}t)$, as $|t|\rightarrow0$ and
	\[
	\lim_{|t|\rightarrow+\infty}\frac{f(t)}{|t|^{p-2}t}=\lim_{|t|\rightarrow+\infty}\frac{g(t)}{|t|^{q-2}t}=+\infty.
	\]
\end{enumerate}
\begin{enumerate}[label=($F_2$),ref=$(F_2)$]
	\item \label{f3}
	There exist $C_{1},C_{2}>0$, $r\in(p,p^{*})$ and $s\in(q,q^{*})$ such that
	\[
	 |f(t)|\leq C_{1}(1+|t|^{r-1}) \quad \mbox{and} \quad |g(t)|\leq C_{2}(1+|t|^{s-1}), \quad \mbox{for all} \hspace{0,2cm} t\in\mathbb{R}.
	\]
\end{enumerate}
\begin{enumerate}[label=($F_3$),ref=$(F_3)$]
	\item \label{f2}
	$t\mapsto\displaystyle\frac{f(t)}{|t|^{p-2}t}$ and $t\mapsto\displaystyle\frac{g(t)}{|t|^{q-2}t}$ are strictly increasing on $|t|\neq0$.
\end{enumerate}
\begin{enumerate}[label=($F_4$),ref=$(F_4)$]
	\item \label{f4}
	$F(t):=\int_{0}^{t}f(\tau)\;\mathrm{d}\tau\leq F(|t|)$ and $G(t):=\int_{0}^{t}g(\tau)\;\mathrm{d}\tau\leq G(|t|)$, for all $t\in\mathbb{R}$.
\end{enumerate}

Under these assumptions we shall consider the energy functional of $C^{1}$ class $I_{\mathrm{o}}$ given by
\begin{align*}
	I_{\mathrm{o}} (u,v)=\frac{1}{p}\|u\|_{a_{\mathrm{o}},p}^{p}+\frac{1}{q}\|v\|_{b_{\mathrm{o}},q}^{q}-\int_{\mathbb{R}^{N}}\left(F(u)+G(v)\right)\;\mathrm{d} x- \int_{\mathbb{R}^{N}}\lambda_{\mathrm{o}}(x)|u|^{\alpha}|v|^{\beta}\;\mathrm{d} x.
\end{align*}
From a standard point of view finding weak solutions to the elliptic problem \eqref{paper2ej0} is equivalent to find critical points for the energy functional $I_{\mathrm{o}}$. In order to get ground state solutions is usual to consider the Nehari method. The standard Nehari manifold for System \eqref{paper2ej0} is defined by
\begin{equation*}
	\mathcal{M}_{\mathrm{o}} = \left\{ (u, v) \in H^{1}(\mathbb{R}^{N}) \times H^{1}(\mathbb{R}^{N})\backslash \{(0,0)\}: \langle I'_{\mathrm{o}}(u,v), (u,v) \rangle = 0 \right\}.
\end{equation*}
In the present work we are interesting in to ensure existence of ground state solutions for the elliptic problem \eqref{paper2ej0} considering $1 < p \leq q < N$. When $p \neq q$ the principal part in the energy functional is not homogeneous. As a consequence the Nehari manifold $\mathcal{M}_{\mathrm{o}}$ is not suitable for our work. The main problem is to guarantee that any Palais-Smale sequence in $\mathcal{M}_{\mathrm{o}}$ is bounded. Another difficulty is to ensure that any nonzero pair $(u,v) \in H^{1}(\mathbb{R}^{N}) \times H^{1}(\mathbb{R}^{N})$ admits a unique projection in the standard Nehari manifold $\mathcal{M}_{\mathrm{o}}$. Furthermore, assuming that $p \neq q$, is not clear whether $\mathcal{M}_{\mathrm{o}}$ is a $C^{1}$ manifold which is crucial in our arguments. In order to overcome these difficulties we shall introduce the following Nehari manifold
\begin{equation*}
	\mathcal{N}_{\mathrm{o}} = \left\{ (u, v) \in H^{1}(\mathbb{R}^{N}) \times H^{1}(\mathbb{R}^{N})\backslash \{(0,0)\}: \left\langle I'_{\mathrm{o}}(u,v), \left(\dfrac{1}{p}u,\dfrac{1}{q}v\right) \right \rangle = 0 \right\}.
\end{equation*}
Here we mention that $\mathcal{N}_{\mathrm{o}}$ is a $C^{1}$ manifold and any Palais-Smale sequence over $\mathcal{N}_{\mathrm{o}}$ is bounded and away from zero, see Lemma \ref{principal} ahead. More precisely, we have that $I_{\mathrm{o}}$ is coercive over $\mathcal{N}_{\mathrm{o}}$. Related to the Nehari manifold we need to consider also the fibering maps which is a powerful tool in the Nehari method. One more time due the loss of homogeneity we introduce the fibering maps $t \rightarrow I_{\mathrm{o}}(t^{1/p}u, t^{1/q}v)$ which coincides with the usual fibering maps only in the case $p = q$. Thanks to the fibering maps we can prove that any nonzero pair $(u,v) \in H^{1}(\mathbb{R}^{N}) \times H^{1}(\mathbb{R}^{N})$ admits a unique projection in the Nehari manifold $\mathcal{N}_{\mathrm{o}}$, see Lemma~\ref{neh}.

It is important to stress that $\lambda : \mathbb{R}^{N} \rightarrow \mathbb{R}$ can be a sign changing continuous function. Therefore the coupled term is an indefinite nonlinear function. This allow us to consider many quasilinear elliptic systems where the coupled term is governed by the term $\lambda(x)|u|^{\alpha}|v|^{\beta}$. Hence is not clear whether the functional $I_{\mathrm{o}}$ has a minimizer over the Nehari manifold. Another difficulty here is to guarantee that any minimizer sequence in the Nehari manifold is bounded. In order to control the behavior of $\lambda(x)$ we consider hypothesis \ref{v2} ensuring that System \eqref{paper2ej0} admits a ground state solution taking into account the fact that the coupling term could be a sign changing function.

Now we are in position in order to state our main first result:
\begin{theorem}\label{A}
	If \ref{v1}-\ref{v2} and \ref{f1}-\ref{f4} hold, then there exists a ground state for System~\eqref{paper2ej0}. Moreover, we have the following conclusions:
	\begin{itemize}
		\item[(i)] Assume also that $\lambda_{\mathrm{0}}(x)\geq0$ for all $x\in\mathbb{R}^{N}$, then there exists a nonnegative ground state for System~\eqref{paper2ej0};
		\item[(ii)] Assume also that \ref{v3'} holds and $\lambda_{\mathrm{0}}(x)\geq0$ for all $x\in\mathbb{R}^{N}$, then there exists a positive ground state for System~\eqref{paper2ej0}, for some $\lambda_{0}>0$.
	\end{itemize}
\end{theorem}

We are also concerned with the existence of positive ground states for the following class of coupled systems
\begin{equation}\label{ej00}
\left\{
\begin{array}{lr}
-\Delta_{p} u+a(x)|u|^{p-2}u=f(u)+ \alpha\lambda(x)|u|^{\alpha-2}u|v|^{\beta}, & x\in\mathbb{R}^{N},\\
-\Delta_{q} v+b(x)|v|^{q-2}v=g(v)+ \beta\lambda(x)|v|^{\beta-2}v|u|^{\alpha}, & x\in\mathbb{R}^{N},
\end{array}
\right. \tag{$S$}
\end{equation}
where the potentials $a(x)$, $b(x)$ and $\lambda(x)$ are asymptotically periodic. Analogously to the periodic case, we introduce the following suitable spaces
\[
E_{a,p}=\left\{u\in W^{1,p}(\mathbb{R}^{N}):\int_{\mathbb{R}^{N}}a(x)|u|^{p}\;\mathrm{d} x<+\infty\right\}, \quad \|u\|_{a,p}^{p}=\int_{\mathbb{R}^{N}}(|\nabla u|^{p}+a(x)|u|^{p})\;\mathrm{d} x,
\]
\[
E_{b,q}=\left\{v\in W^{1,q}(\mathbb{R}^{N}):\int_{\mathbb{R}^{N}}b(x)|v|^{q}\;\mathrm{d} x<+\infty\right\}, \quad \|v\|_{b,q}^{q}=\int_{\mathbb{R}^{N}}(|\nabla v|^{q}+b(x)|u|^{q})\;\mathrm{d} x.
\]
We set the product space $E=E_{a,p}\times E_{b,q}$ endowed with the norm $\|(u,v)\|=\|u\|_{a,p}+\|v\|_{b,q}$. Moreover, we assume the following hypotheses:
\begin{enumerate}[label=($V_{4}$),ref=$(V_{4})$]
	\item \label{v4}
	$a(x)<a_{\mathrm{o}}(x)$, $b(x)<b_{\mathrm{o}}(x)$, $\lambda_{\mathrm{o}}(x)<\lambda(x)$, for all $x\in\mathbb{R}^{N}$ and
	\[
	\lim_{|x|\rightarrow+\infty}|a_{\mathrm{o}}(x)-a(x)|=\lim_{|x|\rightarrow+\infty}|b_{\mathrm{o}}(x)-b(x)|=\lim_{|x|\rightarrow+\infty}|\lambda(x)-\lambda_{\mathrm{o}}(x)|=0.
	\]
\end{enumerate}
\begin{enumerate}[label=($V_{5}$),ref=$(V_{5})$]
	\item \label{v5}
	$a(x),b(x)\geq0$, for all $x\in\mathbb{R}^{N}$ and
	\[
	\inf_{u\in E_{a,p}}\left\{\int_{\mathbb{R}^{N}}|\nabla u|^{p}\;\mathrm{d} x+\int_{\mathbb{R}^{N}}a(x)|u|^{p}\;\mathrm{d} x:\int_{\mathbb{R}^{N}}|u|^{p}\;\mathrm{d} x=1\right\}>0,
	\]
	\[
	\inf_{v\in E_{b,q}}\left\{\int_{\mathbb{R}^{N}}|\nabla v|^{q}\;\mathrm{d} x+\int_{\mathbb{R}^{N}}b(x)|v|^{q}\;\mathrm{d} x:\int_{\mathbb{R}^{N}}|v|^{q}\;\mathrm{d} x=1\right\}>0.
	\]
\end{enumerate}
\begin{enumerate}[label=($V_{6}$),ref=$(V_{6})$]
	\item \label{v6}
	We assume $|\lambda(x)|\leq\delta a(x)^{\alpha/p}b(x)^{\beta/q}$, for some $\delta\in(0,1)$, such that
	\[
	\frac{1}{q}-\delta\max\left\{\frac{\alpha}{p},\frac{\beta}{q}\right\}>0.
	\]
\end{enumerate}
\begin{enumerate}[label=($V_{6}'$),ref=$(V_{6}')$]
	\item \label{v6'}
	We suppose \ref{v6} holds and there exists $R>0$ such that $\lambda(x)\geq\lambda>0$, for all $x\in B_{R}(0)$.
\end{enumerate}
Under these assumptions we are able to state the following result:

\begin{theorem}\label{B}
	If \ref{v1}-\ref{v6} and \ref{f1}-\ref{f4} hold, then there exists a ground state for System~\eqref{ej00}. Moreover, we have the following conclusions:
	\begin{itemize}
		\item[(i)] Assume also that $\lambda(x)\geq0$ for all $x\in\mathbb{R}^{N}$, then there exists a nonnegative ground state for System~\eqref{ej00};
		\item[(ii)] Assume also that \ref{v6'} holds and $\lambda(x)\geq0$ for all $x\in\mathbb{R}^{N}$, then there exists a positive ground state for System~\eqref{ej00}, for some $\lambda>0$.
	\end{itemize}
\end{theorem}

\begin{remark}
	The assumptions \ref{v3} and \ref{v5} imply that the spaces $E_{a_{\mathrm{o},p}}$, $E_{a,p}$ are continuous embedded into $L^{r}(\mathbb{R}^{N})$ for all $r\in[p,p^{*}]$ and the spaces $E_{b_{\mathrm{o},q}}$, $E_{b,q}$ are continuous embedded into $L^{s}(\mathbb{R}^{N})$ for all $s\in[q,q^{*}]$, see \cite[Lemma~2.1]{mana}.
\end{remark}

\begin{remark}
	Recall that the coercive case for the potentials $a(x), b(x)$ have been widely studied by many authors, see \cite{dc,rabinowitz} and references therein. More precisely, we mention that $a(x) \rightarrow +\infty$ and $b(x) \rightarrow + \infty$ as $|x| \rightarrow + \infty$ is said to be the coercive case. In this direction we observe that $E_{a,p}$ and $E_{b,q}$ are Banach spaces.  Furthermore, the embedding $E = E_{a,p} \times E_{b,q} \hookrightarrow L^{s_{1}}(\mathbb{R}^{N}) \times L^{s_{2}}(\mathbb{R}^{N})$ is compact for each $s_{1} \in [p, p^{*})$ and $s_{2} \in [ q, q^{*})$. Under these conditions our main theorems remain true due the compact embedding quoted just above. In fact, any hypothesis on the potentials $a(x)$ and $b(x)$ that ensures the compact embedding, implies that System \eqref{ej0} admits at least one ground state solution via minimization over the Nehari method. For example, we can consider also that for any $M>0$ the set
	$
	\left\{ x \in \mathbb{R}^{N} : a(x) \leq M, \, b(x) \leq M \right\}
	$
	has finite Lebesgue measure. Using this assumption we observe that the compact embedding listed just above holds true, see \cite{bw}.
\end{remark}

\begin{remark}
	Typical examples of nonlinearities satisfying \ref{f1}-\ref{f4} are given by $f(t)=|t|^{p-2}t\ln(1+|t|)$ and $g(t)=|t|^{q-2}t\ln(1+|t|)$. More generally, we can consider also $f(t) = |t|^{p-2}t \ln^{\gamma}(1 + |t|)$ and $g(t) = |t|^{q-2}t \ln^{\gamma}(1 + |t|)$ where $\gamma \geq 1$ is parameter and $p , q > 1$. In these examples the functions satisfy the assumptions \ref{f1}-\ref{f4}. However, these functions does not verify the Ambrosetti-Rabinowitz condition.
\end{remark}

\begin{remark} It is worthwhile to mention that our main results remain true for the following class of quasilinear elliptic systems
	\begin{equation}\label{ed0}
	\left\{
	\begin{array}{lr}
	-\Delta_{p} u+a(x)|u|^{p-2}u = f(u)+ c(x)H_{u}(u,v), & x\in\mathbb{R}^{N},\\
	-\Delta_{q} v+b(x)|v|^{q-2}v  =g(v)+ c(x) H_{v}(u,v), & x\in\mathbb{R}^{N},
	\end{array}
	\right.
	\end{equation}
	where $a(x),b(x)$ are periodic or asymptotically periodic continuous functions. Furthermore, we assume here that $a(x) \geq \ell$ and $b(x) \geq \ell$ for any $x \in \mathbb{R}^{N}$ with $\ell > 0$. Here we also assume that $c \in L^{\infty}(\mathbb{R}^{N})$ and $H: \mathbb{R} \times \mathbb{R} \rightarrow \mathbb{R}$ satisfies the following assumptions:
	\begin{itemize}
		\item[i)] The function $H$ is $C^{1}$ and satisfies a subcritical growth in the following sense
		\[
		|H_{u}(u,v)| \leq c_{1}(1 + |u|^{r_{1} -1} + |v|^{r_{2} -1} ), \quad \mbox{for all} \hspace{0,2cm} (u,v) \in \mathbb{R} \times \mathbb{R},
		\]
		\[
		|H_{v}(u,v)| \leq c_{2}(1 + |u|^{r_{1} -1} + |v|^{r_{2} -1} ), \quad \mbox{for all} \hspace{0,2cm} (u,v) \in \mathbb{R} \times \mathbb{R},
		\]
		for some constants $c_{1} , c_{2} > 0$ and $r_{1} \in (p, p^{*})$, $r_{2} \in (q, q^{*})$;
		\item[ii)] $H(t^{1/p}u, t^{1/q}v) = t H(u,v)$, for any $t \geq 0$ and for all $(u, v) \in \mathbb{R} \times \mathbb{R}$;
		\item[iii)] $|H(u,v)| \leq k (|u|^{p} + |v|^{q})$, for all $(u, v) \in \mathbb{R} \times \mathbb{R}$, where $k > 0$ is small enough.
	\end{itemize}
	The nonlinear terms $f$ and $g$ satisfy the same assumptions discussed in the main theorems.
	Typical examples for $H$ are $H(u,v) = |u|^{\alpha}|v|^{\beta}$ for $(u, v) \in \mathbb{R} \times \mathbb{R}$ where $1 \leq \alpha < p$ and $1 \leq \beta < q$. Here we mention that those more general assumptions over the coupling term can be handle, but for the sake of simplicity, we introduced a particular case given in System \eqref{ej0}.
	Using some minor modifications we can also consider the following elliptic problem
	\begin{equation*}\label{ed1}
	\left\{
	\begin{array}{lr}
	-\Delta_{p} u+a(x)|u|^{p-2}u = R_{u}(u,v) + c(x)H_{u}(u,v), & x\in\mathbb{R}^{N},\\
	-\Delta_{q} v+b(x)|v|^{q-2}v  = R_{v}(u,v)+ c(x) H_{v}(u,v), & x\in\mathbb{R}^{N},
	\end{array}
	\right.
	\end{equation*}
	where $R : \mathbb{R} \times \mathbb{R} \rightarrow \mathbb{R}$ is subcritical and belongs to  $C^{1}$ class. We also mention that in \cite{boc} was proved the existence of solutions for system of quasilinear elliptic equations involving $(p,q)$-Laplacian by studying critical points of the associated energy functional.
\end{remark}

\subsection{Notation}
Let us introduce the following notation:	
\begin{itemize}
	\item $C$, $\tilde{C}$, $C_{1}$, $C_{2}$,... denote positive constants (possibly different).
	\item $o_{n}(1)$ denotes a sequence which converges to $0$ as $n\rightarrow\infty$;
	\item The norm in $L^{s}(\mathbb{R}^{N})$ and $L^{\infty}(\mathbb{R}^{N})$, will be denoted respectively by $\|\cdot\|_{s}$ and $\|\cdot\|_{\infty}$.
	\item The norm in $L^{s}(\mathbb{R}^{N})\times L^{s}(\mathbb{R}^{N})$ is given by $\|(u,v)\|_{s}=\left(\|u\|^{s}_{s}+\|v\|^{s}_{s}\right)^{1/s}$.	
\end{itemize}

\subsection{Outline}
The remainder of this paper is organized as follows: In the forthcoming Section we introduce the variational framework to our problem. In Section \ref{s2} we obtain some preliminary results which will be used throughout the paper. In Section \ref{nehari} we introduce and give some properties of the Nehari manifold associated with the energy functional. In Section~\ref{s4} we use a minimization technique over the Nehari manifold in order to get a nontrivial ground state solution for System~\eqref{paper2ej0}. In this case, we make use of Lion's Lemma and the invariance of the energy functional to obtain the nontrivial critical point. After that, we use the known ground state to get another one which will be nonnegative. By using strong maximum principle we conclude that this ground state will be strictly positive. Finally, in Section~\ref{s5} we study the case when the potentials are asymptotically periodic. For this purpose, we establish a relation between the energy levels associated to Systems~\eqref{paper2ej0} and \eqref{ej00}.

\section{The variational framework}\label{s1}

Associated to System~\eqref{paper2ej0} we have the energy functional $I_{\mathrm{o}}:E_{\mathrm{o}} \rightarrow\mathbb{R}$ given by
\begin{align*}
I_{\mathrm{o}} (u,v)=\frac{1}{p}\|u\|_{a_{\mathrm{o}},p}^{p}+\frac{1}{q}\|v\|_{b_{\mathrm{o}},q}^{q}-\int_{\mathbb{R}^{N}}\left(F(u)+G(v)\right)\;\mathrm{d} x- \int_{\mathbb{R}^{N}}\lambda_{\mathrm{o}}(x)|u|^{\alpha}|v|^{\beta}\;\mathrm{d} x.
\end{align*}
It follows from assumptions \ref{f1} and \ref{f3} that for any $\varepsilon>0$ there is $C_{\varepsilon}>0$ such that
\begin{equation}\label{growth1}
|f(t)|\leq \varepsilon |t|^{p-1}+C_{\varepsilon}|t|^{r-1} \quad \mbox{and} \quad |g(t)|\leq \varepsilon |t|^{q-1}+C_{\varepsilon}|t|^{s-1}, \quad \mbox{for all} \hspace{0,2cm} t\in\mathbb{R},
\end{equation}
which implies that
\begin{equation}\label{growth}
|F(t)|\leq \varepsilon |t|^{p}+C_{\varepsilon}|t|^{r} \quad \mbox{and} \quad |g(t)|\leq \varepsilon |t|^{q}+C_{\varepsilon}|t|^{s}, \quad \mbox{for all} \hspace{0,2cm} t\in\mathbb{R}.
\end{equation}
Using \eqref{growth} one sees that $I_{\mathrm{o}}$ is well defined. Moreover, $I_{\mathrm{o}} \in C^1(E,\mathbb{R})$ and its differential is given~by
\begin{align}
\langle I_{\mathrm{o}}'(u,v),(\phi,\psi)\rangle  = \int_{\mathbb{R}^{N}}(|\nabla u|^{p-2}\nabla u\nabla\phi+a_{\mathrm{o}}(x)|u|^{p-2}u\phi+|\nabla v|^{q-2}\nabla v\nabla\psi+b_{\mathrm{o}}(x)|v|^{q-2}v\psi)\;\mathrm{d} x\nonumber\\
-\int_{\mathbb{R}^{N}}\left(f(u)\phi+g(v)\psi\right)\;\mathrm{d} x- \int_{\mathbb{R}^{N}}\lambda_{\mathrm{o}} (x)(\alpha|u|^{\alpha-2}u|v|^{\beta}\phi+\beta|u|^{\alpha}|v|^{\beta-2}v\psi)\;\mathrm{d} x.\nonumber
\end{align}
Hence, critical points of $I_{\mathrm{o}}$ are precisely the weak solutions of System~\eqref{paper2ej0}.

In order to treat System~\eqref{ej00} variationally, we introduce the $C^{1}$ energy functional $I:E\rightarrow\mathbb{R}$ defined by
\[
I(u,v)=\frac{1}{p}\|u\|_{a,p}^{p}+\frac{1}{q}\|v\|_{b,q}^{q}-\int_{\mathbb{R}^{N}}\left(F(u)+G(v)\right)\;\mathrm{d} x- \int_{\mathbb{R}^{N}}\lambda(x)|u|^{\alpha}|v|^{\beta}\;\mathrm{d} x,
\]
which its differential is given by
\begin{align}
\langle I'(u,v),(\phi,\psi)\rangle  = \int_{\mathbb{R}^{N}}(|\nabla u|^{p-2}\nabla u\nabla\phi+a(x)|u|^{p-2}u\phi+|\nabla v|^{q-2}\nabla v\nabla\psi+b(x)|v|^{q-2}v\psi)\;\mathrm{d} x\nonumber\\
-\int_{\mathbb{R}^{N}}\left(f(u)\phi+g(v)\psi\right)\;\mathrm{d} x- \int_{\mathbb{R}^{N}}\lambda (x)(\alpha|u|^{\alpha-2}u|v|^{\beta}\phi+\beta|u|^{\alpha}|v|^{\beta-2}v\psi)\;\mathrm{d} x.\nonumber
\end{align}
Under our assumptions the energy functional $I$ is well defined and the critical points correspond to solutions of System~\eqref{ej00}.

\section{Preliminary results}\label{s2}

\begin{lemma}\label{nehari-1}
	If \ref{v2} holds, then
	\begin{equation}\label{ej27}
	\int_{\mathbb{R}^{N}}\lambda_{\mathrm{o}}(x)|u|^{\alpha}|v|^{\beta}\;\mathrm{d} x\leq \delta\max\left\{\frac{\alpha}{p},\frac{\beta}{q}\right\}\left(\|u\|_{a_{\mathrm{o}},p}^{p}+\|v\|_{b_{\mathrm{o}},q}^{q}\right), \quad \mbox{for all} \hspace{0,2cm} (u,v)\in E_{\mathrm{o}}.
	\end{equation}
\end{lemma}
\begin{proof}
	In fact, it follows from assumption \ref{v2} that
	\[
	\int_{\mathbb{R}^{N}}\lambda_{\mathrm{o}}(x)|u|^{\alpha}|v|^{\beta}\;\mathrm{d} x\leq \delta\int_{\mathbb{R}^{N}}a_{\mathrm{o}}(x)^{\alpha/p}|u|^{\alpha}b_{\mathrm{o}}(x)^{\beta/q}|v|^{\beta}\;\mathrm{d} x.
	\]
	Since $\alpha/p+\beta/q=1$, we can use Young's inequality to conclude that
	\[
	\int_{\mathbb{R}^{N}}\lambda_{\mathrm{o}}(x)|u|^{\alpha}|v|^{\beta}\;\mathrm{d} x\leq \delta\max\left\{\frac{\alpha}{p},\frac{\beta}{q}\right\}\int_{\mathbb{R}^{N}}\left(a_{\mathrm{o}}(x)|u|^{p}+b_{\mathrm{o}}(x)|v|^{q}\right)\;\mathrm{d} x,
	\]
	which implies \eqref{ej27}.
\end{proof}

\begin{remark}
	We point out that \eqref{ej27} holds true for the asymptotically periodic case.
\end{remark}

\begin{lemma}\label{gs}
	If \ref{f2} holds, then the functions
	\begin{equation}\label{ej17}
	f(t)t-pF(t) \quad \mbox{and} \quad g(t)t-qG(t),
	\end{equation}
	are increasing for $|t|\neq0$. Furthermore, we have
	\begin{equation}\label{ej9}
	f'(t)t^{2}-(p-1)f(t)t>0 \quad \mbox{and} \quad g'(t)t^{2}-(q-1)g(t)t>0,
	\end{equation}
	for all $t\neq0$.
\end{lemma}
\begin{proof}
	In fact, let $0<t_{1}<t_{2}$ be fixed. Thus, by using \ref{f2} we deduce that
	\begin{equation}\label{gs1}
	f(t_{1})t_{1}-pF(t_{1}) < \frac{f(t_{2})}{t_{2}^{p-1}}t_{1}^{p}-pF(t_{2})+p\int_{t_{1}}^{t_{2}}f(\tau)\;\mathrm{d}\tau.
	\end{equation}
	Moreover, we have
	\begin{equation}\label{gs2}
	p\int_{t_{1}}^{t_{2}}f(\tau)\;\mathrm{d}\tau<p\frac{f(t_{2})}{t_{2}^{p-1}}\int_{t_{1}}^{t_{2}}\tau^{p-1}\;\mathrm{d}\tau=\frac{f(t_{2})}{t_{2}^{p-1}}(t_{2}^{p}-t_{1}^{p}).
	\end{equation}
	Combining \eqref{gs1} and \eqref{gs2} we conclude that
	\[
	f(t_{1})t_{1}-pF(t_{1}) < f(t_{2})t_{2}-pF(t_{2}).
	\]
	The same argument can be used to get the result when $t<0$ and for the function $g(t)t-qG(t)$.
	
	Now, we note from \ref{f2} that for $t\in(0,+\infty)$ we have
	\[
	0<\frac{d}{dt}\left(\frac{f(t)}{t^{p-1}}\right)=\frac{f'(t)t^{p-1}-(p-1)f(t)t^{p-2}}{t^{2(p-1)}},
	\]
	\vspace{-0,2cm}
	\[
	0<\frac{d}{dt}\left(\frac{g(t)}{t^{q-1}}\right)=\frac{g'(t)t^{q-1}-(q-1)g(t)t^{q-2}}{t^{2(q-1)}},
	\]
	which implies \eqref{ej9}. Analogously we get the result when $t\in(-\infty,0)$.
\end{proof}

\begin{remark}
	It is important to mention that in view of the preceding Lemma, the functions $f(t)t-pF(t)$ and $g(t)t-qG(t)$ are nonnegative for all $t\in\mathbb{R}$.
\end{remark}

\section{The Nehari manifold}\label{nehari}

Let $\mathcal{N}_{0}$ be the Nehari manifold associated to System~\eqref{paper2ej0} defined by
\[
\mathcal{N}_{\mathrm{o}}:=\left\{(u,v)\in E_{\mathrm{o}}\backslash\{(0,0)\}:\left\langle I_{\mathrm{o}}'(u,v),\left(\frac{1}{p}u,\frac{1}{q}v\right)\right\rangle=0 \right\}.
\]
Hence, $(u,v)\in\mathcal{N}_{\mathrm{o}}$ if and only if satisfies
\begin{equation}\label{ej3}
\frac{1}{p}\|u\|_{a_{\mathrm{o}},p}^{p}+\frac{1}{q}\|v\|_{b_{\mathrm{o}},q}^{q}- \int_{\mathbb{R}^{N}}\lambda_{\mathrm{o}}(x)|u|^{\alpha}|v|^{\beta}\;\mathrm{d} x=\frac{1}{p}\int_{\mathbb{R}^{N}}f(u)u\;\mathrm{d} x+\frac{1}{q}\int_{\mathbb{R}^{N}}g(v)v\;\mathrm{d} x.
\end{equation}

\begin{lemma}\label{principal}
	If \ref{f1}-\ref{f3} hold, then we have the following facts:
	\begin{itemize}
		\item[(i)] $\mathcal{N}_{\mathrm{o}}$ is a $C^{1}$-manifold;
		\item[(ii)] There exists $\gamma>0$ such that $\|(u,v)\|_{\mathrm{o}}\geq\gamma$, for all $(u,v)\in\mathcal{N}_{\mathrm{o}}$.
	\end{itemize}
\end{lemma}
\begin{proof}
	Let $\varphi:E_{\mathrm{o}}\backslash\{(0,0)\}\rightarrow\mathbb{R}$ be defined by $\varphi(u,v)=\langle I_{\mathrm{o}}'(u,v),((1/p)u,(1/q)v)\rangle$. It is no hard to verify that $\varphi$ is in $C^1$ class. Recall also that $\mathcal{N}_{\mathrm{o}}=\varphi^{-1}(0)$. It is suffices to ensure that $0$ is a regular value for the function $\varphi$. Using \eqref{ej9} and \eqref{ej3} we can deduce that
	\[
	\left\langle\varphi'(u,v),\left(\frac{1}{p}u,\frac{1}{q}v\right)\right\rangle\leq -\frac{1}{p^{2}}\int_{\mathbb{R}^{N}}(f'(u)u^{2}-(p-1)f(u)u)-\frac{1}{q^{2}}\int_{\mathbb{R}^{N}}(g'(v)v^{2}-(q-1)g(v)v)<0,
	\]
	which implies that $0$ is a regular value of $\varphi$. Therefore, $\mathcal{N}_{\mathrm{o}}$ is a $C^{1}$-manifold.
	
	In order to prove $(ii)$, we note by Lemma~\ref{nehari-1} that
	\[
	\left(\frac{1}{q}-\delta \max\left\{\frac{\alpha}{p},\frac{\beta}{q}\right\}\right)(\|u\|_{a_{\mathrm{o}},p}^{p}+\|v\|_{b_{\mathrm{o}},q}^{q})\leq \frac{1}{p}\|u\|_{a_{\mathrm{o}},p}^{p}+\frac{1}{q}\|v\|_{b_{\mathrm{o}},q}^{q}- \int_{\mathbb{R}^{N}}\lambda_{\mathrm{o}}(x)|u|^{\alpha}|v|^{\beta}\;\mathrm{d} x.
	\]
	Hence, by using \eqref{growth1} and \eqref{ej3} we can deduce that
	\begin{equation}\label{ej6}
	\left(\frac{1}{q}-\delta \max\left\{\frac{\alpha}{p},\frac{\beta}{q}\right\}\right)(\|u\|_{a_{\mathrm{o}},p}^{p}+\|v\|_{b_{\mathrm{o}},q}^{q})\leq \varepsilon(\|u\|_{a_{\mathrm{o}},p}^{p}+\|v\|_{b_{\mathrm{o}},q}^{q})+\tilde{C_{\varepsilon}}(\|u\|_{a_{\mathrm{o}},p}^{r}+\|v\|_{b_{\mathrm{o}},q}^{s}).
	\end{equation}
	Taking $\varepsilon>0$ sufficiently small such that
	\[
	\left(\frac{1}{q}-\delta \max\left\{\frac{\alpha}{p},\frac{\beta}{q}\right\}-\varepsilon\right)>0,
	\]
	we conclude by \eqref{ej6} that
	\[
	0<\frac{1}{\tilde{C_{\varepsilon}}}\left(\frac{1}{q}-\delta \max\left\{\frac{\alpha}{p},\frac{\beta}{q}\right\}-\varepsilon\right)\leq \|u\|_{a_{\mathrm{o}},p}^{r-p}+\|v\|_{b_{\mathrm{o}},q}^{s-q},
	\]
	which implies $(ii)$.
\end{proof}

\begin{lemma}\label{neh}
	For any $(u,v)\in E_{\mathrm{o}}\backslash\{(0,0)\}$ there exists a unique $t_{0}>0$, depending on $(u,v)$, such that
	\[
	(t_{0}^{1/p}u,t_{0}^{1/q}v)\in\mathcal{N}_{\mathrm{o}} \quad \mbox{and} \quad I_{\mathrm{o}}(t_{0}^{1/p}u,t_{0}^{1/q}v)=\max_{t\geq0}I_{\mathrm{o}}(t^{1/p}u,t^{1/q}v).
	\]
\end{lemma}
\begin{proof}
	Let $(u,v)\in E_{\mathrm{o}}\backslash\{(0,0)\}$ be fixed. We consider $h:[0,+\infty)\rightarrow\mathbb{R}$ defined by $h(t)=I_{\mathrm{o}}(t^{1/p}u,t^{1/q}v)$. Note that
	\[
	h'(t)t=\left\langle I_{\mathrm{o}}'(t^{1/p}u,t^{1/q}v),\left(\frac{1}{p}t^{1/p}u,\frac{1}{q}t^{1/q}v\right)\right\rangle.
	\]
	Thus, $t_{0}$ is a positive critical point of $h$ if and only if $(t_{0}^{1/p}u,t_{0}^{1/q}v)\in\mathcal{N}_{\mathrm{o}}$. Using Lemma~\ref{nehari-1}, the growth conditions of the nonlinearities and Sobolev embedding we can deduce that
	\[
	h(t)\geq t\left[\left(\frac{1}{q}-\delta  \max\left\{\frac{\alpha}{p},\frac{\beta}{q}\right\}-C\varepsilon\right)(\|u\|_{a_{\mathrm{o}},p}^{p}+\|v\|_{b_{\mathrm{o}},q}^{q})-C_{\varepsilon}t^{\frac{r-p}{p}}\|u\|_{a_{\mathrm{o}},p}^{r}-C_{\varepsilon}t^{\frac{s-q}{q}}\|v\|_{b_{\mathrm{o}},q}^{s}\right].
	\]
	Taking $\varepsilon$ sufficiently small, we conclude that $h(t)\geq0$ provided that $t>0$ is small. On the other hand, we can deduce that
	\[
	\frac{h(t)}{t}\leq \frac{1}{p}\|u\|_{a_{\mathrm{o}},p}^{p}+\frac{1}{q}\|v\|_{b_{\mathrm{o}},q}^{q}- \int_{\mathbb{R}^{N}}\lambda_{\mathrm{o}}(x)|u|^{\alpha}|v|^{\beta}\;\mathrm{d} x-\int_{\{u\neq0\}}\frac{F(t^{1/p}u)}{(t^{1/p}|u|)^{p}}|u|^{p}\;\mathrm{d} x-\int_{\{v\neq0\}}\frac{G(t^{1/q}v)}{(t^{1/q}|v|)^{q}}|v|^{q}\;\mathrm{d} x,
	\]
	which together with \ref{f1} implies that $h(t)\leq0$ for $t>0$ large.	Thus, $h$ has maximum points in $(0,+\infty)$. Now, note that every critical point $t\in(0,+\infty)$ of $h$ satisfies
	\begin{equation}\label{ej10}
	\frac{1}{p}\|u\|_{a_{\mathrm{o}},p}^{p}+\frac{1}{q}\|v\|_{b_{\mathrm{o}},q}^{q}- \int_{\mathbb{R}^{N}}\lambda_{\mathrm{o}}(x)|u|^{\alpha}|v|^{\beta}\;\mathrm{d} x=\frac{1}{p}\int_{\mathbb{R}^{N}}\frac{f(t^{1/p}u)u}{t^{1-\frac{1}{p}}}\;\mathrm{d} x+\frac{1}{q}\int_{\mathbb{R}^{N}}\frac{g(t^{1/q}v)v}{t^{1-\frac{1}{q}}}\;\mathrm{d} x.
	\end{equation}
	By using \eqref{ej9}, we have
	\begin{equation}\label{ej30}
	\frac{d}{dt}\left(\frac{f(t^{1/p}u)u}{t^{1-\frac{1}{p}}}\right)=\frac{f'(t^{1/p}u)(t^{1/p}u)^{2}-(p-1)f(t^{1/p}u)t^{1/p}u}{pt^{2-\frac{1}{p}}}>0,
	\end{equation}
	\begin{equation}\label{ej31}	
	\frac{d}{dt}\left(\frac{g(t^{1/q}v)v}{t^{1-\frac{1}{q}}}\right)=\frac{g'(t^{1/q}v)(t^{1/q}v)^{2}-(q-1)g(t^{1/q}v)t^{1/q}v}{qt^{2-\frac{1}{q}}}>0.
	\end{equation}
	Therefore, the right-hand side of \eqref{ej10} is increasing on $t>0$ which implies that the critical point is unique.	
\end{proof}

\section{Proof of Theorem~\ref{A}}\label{s4}

In order to prove Theorem~\ref{A}, we introduce the Nehari energy level associated with System~\eqref{paper2ej0} defined by
\[
c_{\mathcal{N}_{\mathrm{o}}}=\inf_{(u,v)\in\mathcal{N}_{\mathrm{o}}}I_{\mathrm{o}}(u,v).
\]
Let $(u_{n},v_{n})_{n}\subset \mathcal{N}_{\mathrm{o}}$ be a minimizing sequence to $c_{\mathcal{N}_{\mathrm{o}}}$, that is,
\begin{equation}\label{ej12}
I_{\mathrm{o}}(u_{n},v_{n})\rightarrow c_{\mathcal{N}_{\mathrm{o}}} \quad \mbox{and} \quad \left\langle I_{\mathrm{o}}'(u_{n},v_{n}),\left(\frac{1}{p}u_{n},\frac{1}{q}v_{n}\right)\right\rangle=0.
\end{equation}

\begin{proposition}\label{p1}
	The minimizing sequence $(u_{n},v_{n})_{n}$ is bounded in $E_{\mathrm{o}}$.
\end{proposition}
\begin{proof}
	Arguing by contradiction we suppose that $\|(u_{n},v_{n})\|_{\mathrm{o}}=\|u_{n}\|_{a_{\mathrm{o}},p}+\|v_{n}\|_{b_{\mathrm{o}},q}\rightarrow+\infty$, as $n\rightarrow+\infty$. We define $w_{n}=u_{n}/K_{n}^{1/p}$ and $z_{n}=v_{n}/K_{n}^{1/q}$, where $K_{n}:=\|u_{n}\|_{a_{\mathrm{o}},p}^{p}+\|v_{n}\|_{b_{\mathrm{o}},q}^{q}$. Thus,
	\[
	\|w_{n}\|_{a_{\mathrm{o}},p}^{p}+\|z_{n}\|_{b_{\mathrm{o}},q}^{q}=1 \quad \mbox{and} \quad K_{n}\rightarrow+\infty, \hspace{0,2cm} \mbox{as} \hspace{0,2cm} n\rightarrow+\infty.
	\]
	Hence, $(w_{n},z_{n})_{n}$ is bounded in $E_{\mathrm{o}}$. Thus, we may assume up to a subsequence that
	\begin{itemize}
		\item $(w_{n},z_{n})\rightharpoonup(w_{0},z_{0})$ weakly in $E_{\mathrm{o}}$;
		\item $w_{n}\rightarrow w_{0}$ strongly in $L^{r}_{loc}(\mathbb{R}^{N})$, for all $p\leq r<p^{*}$;
		\item $z_{n}\rightarrow z_{0}$ strongly in $L^{s}_{loc}(\mathbb{R}^{N})$, for all $q\leq s<q^{*}$.
		\item $w_{n}(x)\rightarrow w_{0}(x)$ and $z_{n}(x)\rightarrow z_{0}(x)$, almost everywhere in $\mathbb{R}^{N}$.
	\end{itemize}	
	We split the argument into two cases:
	
	\vspace{0,3cm}
	
	\noindent \textbf{Case 1.} $(w_{0},z_{0})\neq(0,0)$.
	
	\vspace{0,3cm}
	
	Let us assume without loss of generality that $w_{0}\neq0$. By using Lemma~\ref{nehari-1} and \eqref{ej12} we can deduce that
	\[
	o_{n}(1)=\frac{I_{\mathrm{o}}(u_{n},v_{n})}{K_{n}}\leq \frac{1}{p}+\delta \max\left\{\frac{\alpha}{p},\frac{\beta}{q}\right\}-\int_{\{u_{n}\neq 0\}}\frac{F(u_{n})}{K_{n}}\;\mathrm{d} x.
	\]
	The last inequality jointly with \ref{f1} and Fatou's Lemma leads to
	\[
	\frac{1}{p}+\delta \max\left\{\frac{\alpha}{p},\frac{\beta}{q}\right\}\geq\int_{\{u_{n}\neq0\}}\liminf_{n\rightarrow+\infty}\frac{F(u_{n})}{|u_{n}|^{p}}|w_{n}|^{p}\;\mathrm{d} x=+\infty,
	\]
	which is a contradiction.
	
	\vspace{0,3cm}
	
	\noindent \textbf{Case 2.} $(w_{0},z_{0})=(0,0)$.
	
	\vspace{0,3cm}
	
	First, we claim that for any $R>0$ we have
	\begin{equation}\label{ej13}
	\lim_{n\rightarrow+\infty}\sup_{y\in\mathbb{R}^{N}}\int_{B_{R}(y)}(|w_{n}|^{p}+|z_{n}|^{q})\;\mathrm{d} x=0.
	\end{equation}
	In fact, if \eqref{ej13} does not holds then there exist $R,\eta>0$ such that
	\[
	\lim_{n\rightarrow+\infty}\sup_{y\in\mathbb{R}^{N}}\int_{B_{R}(y)}(|w_{n}|^{p}+|z_{n}|^{q})\;\mathrm{d} x\geq\eta>0.
	\]
	Hence, we can consider a sequence $(y_{n})_{n}\subset\mathbb{Z}^{N}$ such that
	\[
	\lim_{n\rightarrow+\infty}\int_{B_{R}(y_{n})}(|w_{n}|^{p}+|z_{n}|^{q})\;\mathrm{d} x\geq\frac{\eta}{2}>0.
	\]
	We define the shift sequence $(\tilde{w}_{n}(x),\tilde{z}_{n}(x))=(w_{n}(x+y_{n}),z_{n}(x+y_{n}))$. Since $a_{\mathrm{o}}(\cdot)$ and $b_{\mathrm{o}}(\cdot)$ are periodic, we have $\|(w_{n},z_{n})\|_{\mathrm{o}}=\|(\tilde{w}_{n},\tilde{z}_{n})\|_{\mathrm{o}}$. Thus, up to a subsequence, we may assume that
	\begin{itemize}
		\item $(\tilde{w}_{n},\tilde{z}_{n})\rightharpoonup(\tilde{w}_{0},\tilde{z}_{0})$ weakly in $E_{\mathrm{o}}$;
		\item $\tilde{w}_{n}\rightarrow \tilde{w}_{0}$ strongly in $L^{r}_{loc}(\mathbb{R}^{N})$, for all $p\leq r<p^{*}$;
		\item $\tilde{z}_{n}\rightarrow \tilde{z}_{0}$ strongly in $L^{s}_{loc}(\mathbb{R}^{N})$, for all $q\leq s<q^{*}$.
	\end{itemize}	
	Then, we have
	\[
	\lim_{n\rightarrow+\infty}\int_{B_{R}(0)}(|\tilde{w}_{n}|^{p}+|\tilde{z}_{n}|^{q})\;\mathrm{d} x=\lim_{n\rightarrow+\infty}\int_{B_{R}(y_{n})}(|w_{n}|^{p}+|z_{n}|^{q})\;\mathrm{d} x\geq\frac{\eta}{2}>0,
	\]
	which implies that $(\tilde{w}_{0},\tilde{z}_{0})\neq(0,0)$. Arguing as in \textbf{Case 1} we get a contradiction.
	
	Since \eqref{ej13} holds, it follows from \cite[Lemma~1.21]{will} (see also \cite{lionss}) that
	\begin{equation}\label{lionn}
	\lim_{n\rightarrow+\infty}\int_{\mathbb{R}^{N}}|w_{n}|^{r}\;\mathrm{d} x=0 \quad \mbox{and} \quad  \lim_{n\rightarrow+\infty}\int_{\mathbb{R}^{N}}|z_{n}|^{s}\;\mathrm{d} x=0.
	\end{equation}
	By using \eqref{growth} and \eqref{lionn}, we can conclude that
	\begin{equation}\label{ej16}
	\lim_{n\rightarrow+\infty}\int_{\mathbb{R}^{N}}F(\xi^{1/p} w_{n})\;\mathrm{d} x=\lim_{n\rightarrow+\infty}\int_{\mathbb{R}^{N}}G(\xi^{1/q} z_{n})\;\mathrm{d} x=0, \quad \mbox{for all} \hspace{0,2cm} \xi>0.
	\end{equation}
	Since $(u_{n},v_{n})_{n}\subset\mathcal{N}_{\mathrm{o}}$, it follows from Lemma~\ref{neh} that
	\begin{equation}\label{ej18}
	I_{\mathrm{o}}(u_{n},v_{n})\geq I_{\mathrm{o}}(t^{1/p}u_{n},t^{1/q}v_{n}), \quad \mbox{for all} \hspace{0,2cm} t\geq0.
	\end{equation}
	Taking $t=\xi/K_{n}$ and combining \eqref{ej16} and \eqref{ej18} we deduce that
	\[
	c_{\mathcal{N}_{\mathrm{o}}}+o_{n}(1)=I_{\mathrm{o}}(u_{n},v_{n})\geq I_{\mathrm{o}}(\xi^{1/p} w_{n},\xi^{1/q} z_{n})\geq \left(\frac{1}{q}-\delta \max\left\{\frac{\alpha}{p},\frac{\beta}{q}\right\}\right)\xi+o_{n}(1),
	\]
	which is a contradiction for $\xi>0$ sufficiently large. Therefore, $(u_{n},v_{n})_{n}$ is bounded in $E_{\mathrm{o}}$.
\end{proof}

\begin{remark}
Using the same ideas discussed in the proof of Proposition \ref{p1} we mention that the energy functional $I_{\mathrm{o}}$
is coercive over the Nehari manifold $\mathcal{N}_{\mathrm{o}}$.
\end{remark}

In view of Proposition~\ref{p1} we may assume, up to a subsequence, that
\begin{itemize}
	\item $(u_{n},v_{n})\rightharpoonup(u_{0},v_{0})$ weakly in $E_{\mathrm{o}}$;
	\item $u_{n}\rightarrow u_{0}$ strongly in $L^{r}_{loc}(\mathbb{R}^{N})$, for all $p\leq r<p^{*}$;
	\item $v_{n}\rightarrow v_{0}$ strongly in $L^{s}_{loc}(\mathbb{R}^{N})$, for all $q\leq s<q^{*}$;
	\item $u_{n}(x)\rightarrow u_{0}(x)$ and $v_{n}(x)\rightarrow v_{0}(x)$, almost everywhere in $\mathbb{R}^{N}$.
\end{itemize}
Since $C^{\infty}_{0}(\mathbb{R}^{N})\times C^{\infty}_{0}(\mathbb{R}^{N})$ is dense into the space $E_{\mathrm{o}}$, it follows by standard arguments that $I_{\mathrm{o}}'(u_{0},v_{0})=0$, that is, $(u_{0},v_{0})$ is a solution for System~\eqref{paper2ej0}. In order to get a nontrivial solution, we shall prove the following result:

\begin{proposition}\label{p2}
	Let $(u_{n},v_{n})_{n}\subset \mathcal{N}_{\mathrm{o}}$ be the minimizing sequence satisfying \eqref{ej12}. Then, there exists a sequence $(y_{n})_{n}\subset\mathbb{R}^{N}$ and constants $R,\eta>0$ such that $|y_{n}|\rightarrow\infty$ as $n\rightarrow\infty$, and
	\begin{equation}\label{vanish}
	\liminf_{n\rightarrow+\infty}\int_{B_{R}(y_{n})}(|u_{n}|^{p}+|v_{n}|^{q})\;\mathrm{d} x\geq\eta>0.
	\end{equation}
\end{proposition}

\begin{proof}
	Arguing by contradiction, we suppose that \eqref{vanish} does not hold. Then we have
	\[
	\lim_{n\rightarrow\infty}\sup_{y\in\mathbb{R}^{N}}\int_{B_{R}(y)}|u_{n}|^{p}\;\mathrm{d} x=0 \quad \mbox{and} \quad  \lim_{n\rightarrow\infty}\sup_{y\in\mathbb{R}^{N}}\int_{B_{R}(y)}|v_{n}|^{q}\;\mathrm{d} x=0,
	\]
	for any $R>0$. Hence, we conclude that
	\begin{equation}\label{lion}
	\lim_{n\rightarrow+\infty}\int_{\mathbb{R}^{N}}|u_{n}|^{r}\;\mathrm{d} x=0 \quad \mbox{and} \quad  \lim_{n\rightarrow+\infty}\int_{\mathbb{R}^{N}}|v_{n}|^{s}\;\mathrm{d} x=0.
	\end{equation}
	Using \eqref{growth1} and Lemma~\ref{nehari-1}, we can deduce that
	\begin{eqnarray}
	0 & = & \left\langle I_{\mathrm{o}}'(u_{n},v_{n}),\left(\frac{1}{p}u_{n},\frac{1}{q}v_{n}\right)\right\rangle\nonumber\\
	& \geq & \left(\frac{1}{q}-\delta \max\left\{\frac{\alpha}{p},\frac{\beta}{q}\right\}\right)(\|u_{n}\|_{a_{\mathrm{o}},p}^{p}+\|v_{n}\|_{b_{\mathrm{o}},q}^{q})-\frac{1}{p}\int_{\mathbb{R}^{N}}f(u_{n})u_{n}\;\mathrm{d} x-\frac{1}{q}\int_{\mathbb{R}^{N}}g(v_{n})v_{n}\;\mathrm{d} x\nonumber\\
	& \geq & \left(\frac{1}{q}-\delta \max\left\{\frac{\alpha}{p},\frac{\beta}{q}\right\}-\varepsilon\right)(\|u_{n}\|_{a_{\mathrm{o}},p}^{p}+\|v_{n}\|_{b_{\mathrm{o}},q}^{q})-C_{\varepsilon}(\|u_{n}\|_{r}^{r}+\|v_{n}\|_{s}^{s}).\label{ej15}
	\end{eqnarray}
	Taking $\varepsilon>0$ sufficiently small such that
	\[
	\frac{1}{q}-\delta \max\left\{\frac{\alpha}{p},\frac{\beta}{q}\right\}-\varepsilon>0,
	\]
	it follows from \eqref{lion} and \eqref{ej15} that
	\[
	0\geq \left(\frac{1}{q}-\delta \max\left\{\frac{\alpha}{p},\frac{\beta}{q}\right\}-\varepsilon\right)(\|u_{n}\|_{a_{\mathrm{o}},p}^{p}+\|v_{n}\|_{b_{\mathrm{o}},q}^{q})+o_{n}(1),
	\]
	which implies that $\|(u_{n},v_{n})\|_{\mathrm{o}}\rightarrow0$ as $n\rightarrow+\infty$. However, since $I_{\mathrm{o}}(u_{n},v_{n})\rightarrow c_{\mathcal{N}_{\mathrm{o}}}>0$ and $I_{\mathrm{o}}$ is continuous, the minimizing sequence $(u_{n},v_{n})_{n}$ can not converge to zero strongly in $E_{\mathrm{o}}$. Therefore, this contradiction implies that \eqref{vanish} holds.	
\end{proof}

\begin{proposition}\label{p3}
	There exists a ground state solution for System~\eqref{paper2ej0}.
\end{proposition}
\begin{proof}
	Let $(u_{0},v_{0})$ be the critical point of the energy functional $I$. We split the proof into two cases.	
	\vspace{0,3cm}
	
	\noindent \textbf{Case 1.} $(u_{0},v_{0})\neq(0,0)$.
	
	\vspace{0,3cm}
	
	\noindent If $(u_{0},v_{0})\neq(0,0)$, then we have a nontrivial solution for System~\eqref{paper2ej0}. It remains to prove that $(u_{0},v_{0})$ is in fact a ground state. We note that $(u_{0},v_{0})\in\mathcal{N}_{\mathrm{o}}$. Thus, $c_{\mathcal{N}_{\mathrm{o}}}\leq I_{\mathrm{o}}(u_{0},v_{0})$. On the other hand, using \eqref{ej17}, \eqref{ej12} and Fatou's Lemma, we can deduce that
	\begin{eqnarray*}
		c_{\mathcal{N}_{\mathrm{o}}}+o_{n}(1) & = & I_{\mathrm{o}}(u_{n},v_{n})-\left\langle I_{\mathrm{o}}'(u_{n},v_{n}),\left(\frac{1}{p}u_{n},\frac{1}{q}v_{n}\right)\right\rangle\\
		& = & \frac{1}{p}\int_{\mathbb{R}^{N}}(f(u_{n})u_{n}-pF(u_{n}))\;\mathrm{d} x+\frac{1}{q}\int_{\mathbb{R}^{N}}(g(v_{n})v_{n}-qG(v_{n}))\;\mathrm{d} x\\
		& \geq & \frac{1}{p}\int_{\mathbb{R}^{N}}(f(u_{0})u_{0}-pF(u_{0}))\;\mathrm{d} x+\frac{1}{q}\int_{\mathbb{R}^{N}}(g(v_{0})v_{0}-qG(v_{0}))\;\mathrm{d} x+o_{n}(1)\\
		& = & I_{\mathrm{o}}(u_{0},v_{0})-\left\langle I_{\mathrm{o}}'(u_{0},v_{0}),\left(\frac{1}{p}u_{0},\frac{1}{q}v_{0}\right)\right\rangle+o_{n}(1)\\
		& = & I_{\mathrm{o}}(u_{n},v_{n})+o_{n}(1),
	\end{eqnarray*}
	which implies that $c_{\mathcal{N}_{\mathrm{o}}}\geq I_{\mathrm{o}}(u_{0},v_{0})$. Therefore, $I_{\mathrm{o}}(u_{0},v_{0})=c_{\mathcal{N}_{\mathrm{o}}}$, that is, $(u_{0},v_{0})$ is a ground state solution for System~\eqref{paper2ej0}.
	
	\vspace{0,3cm}
	
	\noindent \textbf{Case 2.} $(u_{0},v_{0})=(0,0)$.
	
	\vspace{0,3cm}
	
	\noindent In light of Proposition~\ref{p2}, there exist a sequence $(y_{n})_{n}\subset\mathbb{R}^{N}$ and constants $R,\eta>0$ such that
	\begin{equation}\label{ej4}
	\liminf_{n\rightarrow+\infty}\int_{B_{R}(y_{n})}(|u_{n}|^{p}+|v_{n}|^{q})\;\mathrm{d} x\geq\eta>0.
	\end{equation}
	Without any loss of generality we assume that $(y_{n})_{n} \subset \mathbb{Z}^{N}$. Let us define the shift sequence $(\tilde{u}_{n}(x),\tilde{v}_{n}(x))=(u_{n}(x+y_{n}),v_{n}(x+y_{n}))$. Since $a_{\mathrm{o}}(\cdot)$, $b_{\mathrm{o}}(\cdot)$ and $\lambda_{\mathrm{o}}(\cdot)$ are periodic, we can use the invariance of the energy functional $I_{\mathrm{o}}$, to deduce that
	\[
	\|(u_{n},v_{n})\|_{\mathrm{o}}=\|(\tilde{u}_{n},\tilde{v}_{n})\|_{\mathrm{o}} \quad \mbox{and} \quad I_{\mathrm{o}}(u_{n},v_{n})=I_{\mathrm{o}}(\tilde{u}_{n},\tilde{v}_{n})\rightarrow c_{\mathcal{N}_{\mathrm{o}}}.
	\]
	Moreover, arguing as before, we can conclude that $(\tilde{u}_{n},\tilde{v}_{n})_{n}$ is a bounded sequence in $E_{\mathrm{o}}$. Thus, up to a subsequence, we may assume that
	\begin{itemize}
		\item $(\tilde{u}_{n},\tilde{v}_{n})\rightharpoonup(\tilde{u}_{0},\tilde{v}_{0})$ weakly in $E_{\mathrm{o}}$;
		\item $\tilde{u}_{n}\rightarrow \tilde{u}_{0}$ strongly in $L^{r}_{loc}(\mathbb{R}^{N})$, for all $p\leq r<p^{*}$;
		\item $\tilde{v}_{n}\rightarrow \tilde{v}_{0}$ strongly in $L^{s}_{loc}(\mathbb{R}^{N})$, for all $q\leq s<q^{*}$.
	\end{itemize}
	Moreover, $(\tilde{u},\tilde{v})$ is a critical point of $I_{\mathrm{o}}$. By using \eqref{ej4} one sees that
	\[
	\liminf_{n\rightarrow\infty}\int_{B_{R}(0)}(|\tilde{u}_{n}|^{p}+|\tilde{v}_{n}|^{q})\;\mathrm{d} x=\liminf_{n\rightarrow\infty}\int_{B_{R}(y_{n})}(|u_{n}|^{p}+|v_{n}|^{q})\;\mathrm{d} x\geq\eta>0.
	\]
	Therefore, $(\tilde{u},\tilde{v})\neq(0,0)$ is a solution for System~\eqref{paper2ej0}. The conclusion follows from \textbf{Case 1}.
\end{proof}

\begin{proposition}\label{p4}
	If \ref{f4} holds and $\lambda_{\mathrm{o}}(x)\geq0$ for all $x\in\mathbb{R}^{N}$, then there exists a nonnegative ground state for System~\eqref{paper2ej0}.
\end{proposition}
\begin{proof}
	Let $(u_{0},v_{0})$ be the ground state solution obtained in Proposition~\ref{p3}. Then, from Lemma~\ref{neh} there exists a unique $t_{0}>0$ such that $(t_{0}^{1/p}|u_{0}|,t_{0}^{1/q}|v_{0}|)\in\mathcal{N}_{\mathrm{o}}$. Since $\lambda_{\mathrm{o}}(x)\geq0$, it follows from \ref{f4} that $I_{\mathrm{o}}(t_{0}^{1/p}|u_{0}|,t_{0}^{1/q}|v_{0}|)\leq I_{\mathrm{o}}(t_{0}^{1/p}u_{0},t_{0}^{1/q}v_{0})$. Thus, since $(u_{0},v_{0})\in\mathcal{N}_{\mathrm{o}}$ we have
	\[
	I_{\mathrm{o}}(t_{0}^{1/p}|u_{0}|,t_{0}^{1/q}|v_{0}|)\leq \max_{t\geq0}I_{\mathrm{o}}(t^{1/p}u_{0},t^{1/q}v_{0})=I_{\mathrm{o}}(u_{0},v_{0})=c_{\mathcal{N}_{\mathrm{o}}}.
	\]
	Therefore, $(t_{0}^{1/p}|u_{0}|,t_{0}^{1/q}|v_{0}|)\in\mathcal{N}_{\mathrm{o}}$ is a nonnegative ground state solution for System~\eqref{paper2ej0}.
\end{proof}

At this point, we have obtained a nonnegative ground state solution $(u,v)\in E_{\mathrm{o}}$ for System~\eqref{paper2ej0}. However, this solution could be semitrivial, that is, $(u,0)$ or $(0,v)$. The next step is to prove that if \ref{v3'} holds, then for some $\lambda_{0}>0$ the ground state can not be semitrivial.

\begin{proposition}\label{p7}
	Suppose that \ref{v3'} holds. There exists $\lambda_{0}>0$ such that if $(u,v)\in E_{\mathrm{o}}$ is a ground state for System~\eqref{paper2ej0}, then $u\neq0$ and $v\neq0$.
\end{proposition}
\begin{proof}
	If we consider $\lambda_{\mathrm{o}}(x)=0$, for all $x\in\mathbb{R}^{N}$, then we have the uncoupled equation
	\begin{equation}\label{edj1}
	-\Delta_{p}u+a_{\mathrm{o}}(x)|u|^{p-2}u=f(u), \quad x\in\mathbb{R}^{N}. \tag{$S_{a_{\mathrm{o}}}$}
	\end{equation}
	Let $I_{a_{\mathrm{o}}}:E_{a_{\mathrm{o}},p}\rightarrow\mathbb{R}$ be the energy functional associated to \eqref{edj1} defined by
	\[
	I_{a_{\mathrm{o}}}(u)=\frac{1}{p}\|u\|_{a_{\mathrm{o}},p}^{p}-\int_{\mathbb{R}^{N}}F(u)\;\mathrm{d} x.
	\]
	The Nehari manifold associated to \eqref{edj1} is given by
	\[
	\mathcal{N}_{a_{\mathrm{o}}}=\left\{u\in E_{a_{\mathrm{o}},p}\backslash\{0\}:\langle I_{a_{\mathrm{o}}}^{\prime}(u),u\rangle=0\right\}.
	\]
	Note that the same arguments used in this work holds true for equation \eqref{edj1}. Thus, let $u_{0}\in\mathcal{N}_{a_{\mathrm{o}}}$ be a positive ground state solution for equation \eqref{edj1}. By similar arguments used in the proof of Lemma~\ref{neh} we can deduce that:
	\begin{itemize}
		\item $I_{a_{\mathrm{o}}}(tu_{0})$ is increasing for $0<t<1$;
		\item $I_{a_{\mathrm{o}}}(tu_{0})$ is decreasing for $t>1$;
		\item $I_{a_{\mathrm{o}}}(tu_{0})\rightarrow-\infty$, as $t\rightarrow+\infty$.
	\end{itemize}
	Therefore, $\max_{t\geq0}I_{a_{\mathrm{o}}}(tu_{0})=I_{a_{\mathrm{o}}}(u_{0})$. Analogously, we can introduce $I_{b_{\mathrm{o}}}$, $\mathcal{N}_{b_{\mathrm{o}}}$ and conclude that there exists a positive ground state solution $v_{0}\in\mathcal{N}_{b_{\mathrm{o}}}$ for the uncoupled equation
	\begin{equation}\label{edj2}
	-\Delta_{q}u+b_{\mathrm{o}}(x)|v|^{q-2}v=g(v), \quad x\in\mathbb{R}^{N}. \tag{$S_{b_{\mathrm{o}}}$}
	\end{equation}
	Moreover, $\max_{t\geq0}I_{b_{\mathrm{o}}}(tv_{0})=I_{b_{\mathrm{o}}}(v_{0})$. It follows from Lemma~\ref{neh} that there exists $t_{0}>0$ such that $(t_{0}^{1/p}u_{0},t_{0}^{1/q}v_{0})\in\mathcal{N}_{\mathrm{o}}$. Hence, using \ref{v3'} we can deduce that
	\[
	c_{\mathcal{N}_{\mathrm{o}}} \leq I_{\mathrm{o}}(t_{0}^{1/p}u_{0},t_{0}^{1/q}v_{0})\\
	\leq t_{0}\left(\frac{1}{p}\|u_{0}\|_{a_{\mathrm{o}},p}^{p}+\frac{1}{q}\|v_{0}\|_{b_{\mathrm{o}},q}^{q}- \lambda_{0}\int_{B_{R}(0)}u_{0}^{\alpha}v_{0}^{\beta}\;\mathrm{d} x\right).
	\]
	Thus, for some $\lambda_{0}>0$ we have $c_{\mathcal{N}_{\mathrm{o}}}<\min\{c_{\mathcal{N}_{a_{\mathrm{o}}}},c_{\mathcal{N}_{b_{\mathrm{o}}}}\}$. Therefore, if $I_{\mathrm{o}}(u,v)=c_{\mathcal{N}_{\mathrm{o}}}$, then we have $u\neq0$ and $v\neq0$.
\end{proof}

\begin{proposition}\label{p5}
	If \ref{v3'} holds for suitable $\lambda_{0}>0$, then there exists a positive ground state for System~\eqref{paper2ej0}.
\end{proposition}
\begin{proof}
	According to Proposition \ref{p4} we obtain a nonnegative ground state solution $(u , v)$ for the problem \eqref{paper2ej0}. By using standard arguments for regularity of weak solutions for quasilinear elliptic equations, we have that the functions $u, v$ belong to $C^{1,\alpha}$ for some $\alpha \in (0,1)$, that is, we know that $u, v$ are H\"{o}lder continuous functions, see \cite{Lieber,Lieber2}. It follows from Proposition~\ref{p3} that $(u, v)$ is not trivial. Moreover, in view of Proposition \ref{p7}, the pair $(u, v)$ is not semitrivial, that is, the sets $\{ x \in \mathbb{R}^{N} : u(x) = 0 \}$ and $\{ x \in \mathbb{R}^{N} : v(x) = 0 \}$ are different from the whole space $\mathbb{R}^{N}$. Thus, we have concluded that
	\begin{equation*}
	\left\{
	\begin{array}{lr}
	-\Delta_{p} u +a_{0}(x)u^{p-1} \geq 0, & x\in\mathbb{R}^{N},\\
	u \in E_{a_{0},p} \cap C^{1,\alpha}, \ u \neq 0,
	\end{array}
	\right.
	\end{equation*}
	and
	\begin{equation*}
	\left\{
	\begin{array}{lr}
	-\Delta_{q} v + b_{0}(x)v^{p-1} \geq 0, & x\in\mathbb{R}^{N},\\
	v \in E_{b_{0},q} \cap C^{1,\alpha}, \ v \neq 0.
	\end{array}
	\right.
	\end{equation*}
	Here we mention that $s \rightarrow \beta_{1}(s) := a_{0}(x)s^{p -1}$ and $s \rightarrow \beta_{2}(x) := b_{0}(x)s^{q -1}$ are nondecreasing functions for each  $s > 0$ and $x \in \mathbb{R}^{N}$. By applying the Strong Maximum Principle \cite{pucci} we infer that $u > 0$ and $v > 0$ in $\mathbb{R}^{N}$. This ends the proof.
\end{proof}

\begin{proof}[Proof of Theorem~\ref{A}]
	It follows from Propositions~\ref{p3},~\ref{p4}, \ref{p7} and \ref{p5}.
\end{proof}

\section{Proof of Theorem~\ref{B}}\label{s5}

In this Section we are concerned with the existence of ground states for System~\eqref{ej00}, when the potentials are asymptotically periodic. Analogously to the periodic case, we introduce the Nehari manifold associated to System~\eqref{ej00} defined by
\[
\mathcal{N}:=\left\{(u,v)\in E\backslash\{(0,0)\}:\left\langle I^{\prime}(u,v),\left(\frac{1}{p}u,\frac{1}{q}v\right)\right\rangle\right\},
\]
and the ground state energy $c_{\mathcal{N}}:=\inf_{(u,v)\in\mathcal{N}}I(u,v)$. We point out that all results obtained in Section~\ref{nehari} remains true in the asymptotically periodic case. Thus, $\mathcal{N}$ is a $C^{1}$-manifold and for any $(u,v)\in E\backslash\{(0,0)\}$ there exists a unique $t_{0}>0$, depending only on $(u,v)$, such that
\begin{equation}\label{ej19}
(t_{0}^{1/p}u,t_{0}^{1/q}v)\in\mathcal{N} \quad \mbox{and} \quad I(t_{0}^{1/p}u,t_{0}^{1/q}v)=\max_{t\geq0}I(t^{1/p}u,t^{1/q}v).
\end{equation}
In order to get a ground state solution for \eqref{ej00} we establish a relation between the energy levels $c_{\mathcal{N}_{\mathrm{o}}}$ and $c_{\mathcal{N}}$.

\begin{lemma}\label{rel}
	$c_{\mathcal{N}}<c_{\mathcal{N}_{\mathrm{o}}}$.
\end{lemma}
\begin{proof}
	Let $(u,v)\in\mathcal{N}_{\mathrm{o}}$ be the nonnegative ground state solution for System~\eqref{paper2ej0} obtained in the preceding Section. In light of assumption \ref{v4}, we can deduce that
	\begin{equation}\label{ej20}
	\int_{\mathbb{R}^{N}}\left[(a(x)-a_{\mathrm{o}}(x))u^{p}+(b(x)-b_{\mathrm{o}}(x))v^{q}+(\lambda_{\mathrm{o}}(x)-\lambda(x))uv\right]\;\mathrm{d} x<0,
	\end{equation}
	By using \eqref{ej19} we get a $t_{0}>0$ such that $(t_{0}^{1/p}u,t_{0}^{1/q}v)\in\mathcal{N}$. Thus, it follows from \eqref{ej20} that
	\[
	I(t_{0}^{1/p}u,t_{0}^{1/q}v)-I_{\mathrm{o}}(t_{0}^{1/p}u,t_{0}^{1/q}v)<0.
	\]
	Therefore, since $(u,v)\in\mathcal{N}_{\mathrm{o}}$ we conclude that
	\[
	c_{\mathcal{N}}\leq I(t_{0}^{1/p}u,t_{0}^{1/q}v)<I_{\mathrm{o}}(t_{0}^{1/p}u,t_{0}^{1/q}v)\leq\max_{t\geq0}I_{\mathrm{o}}(t^{1/p}u,t^{1/q}v)=I_{\mathrm{o}}(u,v)=c_{\mathcal{N}_{\mathrm{o}}},
	\]
	which finishes the proof.	
\end{proof}

Let us consider a minimizing sequence $(u_{n},v_{n})_{n}\subset\mathcal{N}$ to $c_{\mathcal{N}}$, that is
\begin{equation}\label{ej21}
I(u_{n},v_{n})\rightarrow c_{\mathcal{N}} \quad \mbox{and} \quad \left\langle I'(u_{n},v_{n}),\left(\frac{1}{p}u_{n},\frac{1}{q}v_{n}\right)\right\rangle=0.
\end{equation}

\begin{proposition}\label{p6}
	The minimizing sequence $(u_{n},v_{n})_{n}$ is bounded in $E$.
\end{proposition}
\begin{proof}
	The proof is similar to the proof of Proposition~\ref{p1} but for the sake of simplicity we give a sketch here. Arguing by contradiction we suppose that $\|(u_{n},v_{n})\|=\|u_{n}\|_{a,p}+\|v_{n}\|_{b,q}\rightarrow+\infty$, as $n\rightarrow+\infty$. We define $w_{n}=u_{n}/K_{n}^{1/p}$ and $z_{n}=v_{n}/K_{n}^{1/q}$, where $K_{n}:=\|u_{n}\|_{a,p}^{p}+\|v_{n}\|_{b,q}^{q}$. Thus, $(w_{n},z_{n})_{n}$ is bounded in $E$. We may assume up to a subsequence that $(w_{n},z_{n})\rightharpoonup(w_{0},z_{0})$ weakly in $E$. If $(w_{0},z_{0})\neq(0,0)$, then we get a contradiction as the same way to \textbf{Case 1} in Proposition~\ref{p1}. If $(w_{0},v_{0})=(0,0)$, then we claim that for any $R>0$ we have
	\begin{equation}\label{ej22}
	\lim_{n\rightarrow+\infty}\sup_{y\in\mathbb{R}^{N}}\int_{B_{R}(y)}(|w_{n}|^{p}+|z_{n}|^{q})\;\mathrm{d} x=0.
	\end{equation}
	If \eqref{ej22} does not hold, then there exist a sequence $(y_{n})_{n}\subset\mathbb{Z}^{N}$ and $R,\eta>0$ such that
	\begin{equation}\label{ej23}
	\lim_{n\rightarrow+\infty}\int_{B_{R}(y_{n})}(|w_{n}|^{p}+|z_{n}|^{q})\;\mathrm{d} x\geq\eta>0.
	\end{equation}
	We define the shift sequence $(\tilde{w}_{n}(x),\tilde{z}_{n}(x))=(w_{n}(x+y_{n}),z_{n}(x+y_{n}))$. Since $E_{a}\hookrightarrow W^{1,p}(\mathbb{R}^{N})$ and $E_{b}\hookrightarrow W^{1,q}(\mathbb{R}^{N})$, we deduce that
	\begin{eqnarray*}
		\|(\tilde{w}_{n},\tilde{z}_{n})\| & = & \int_{\mathbb{R}^{N}}\left(|\nabla \tilde{w}_{n}(x)|^{p}+a(x)|\tilde{w}_{n}(x)|^{p}\right)\;\mathrm{d} x+\int_{\mathbb{R}^{N}}\left(|\nabla \tilde{z}_{n}(x)|^{q}+b(x)|\tilde{z}_{n}(x)|^{q}\right)\;\mathrm{d} x\\
		& \leq & \max\{1,\|a\|_{\infty}\}\|w_{n}\|_{W^{1,p}(\mathbb{R}^{N})}^{p}+\max\{1,\|b\|_{\infty}\}\|z_{n}\|_{W^{1,p}(\mathbb{R}^{N})}^{q}\\
		& \leq & C\|(w_{n},z_{n})\|^{p} + C \|(w_{n},z_{n})\|^{q},
	\end{eqnarray*}
	which implies that $(\tilde{w}_{n},\tilde{z}_{n})_{n}$ is bounded in $E$ due the fact that $(w_{n}, z_{n})_{n}$ is bounded. Thus up to a subsequence that $(\tilde{w}_{n},\tilde{z}_{n})\rightharpoonup(\tilde{w}_{0},\tilde{z}_{0})$. By using \eqref{ej23} we conclude that $(\tilde{w}_{0},\tilde{z}_{0})\neq(0,0)$ and we get a contradiction as in \textbf{Case 1}. Therefore, \eqref{ej22} holds and the conclusion follows as in \textbf{Case 2} of Proposition~\ref{p1}.
\end{proof}

In view of the preceding Proposition, we may assume, up to a subsequence, that $(u_{n},v_{n})\rightharpoonup(u_{0},v_{0})$ weakly in $E$. By a standard density argument we can conclude that $(u_{0},v_{0})$ is a critical point of $I$. The main difficulty here is to prove that $(u_{0},v_{0})$ is a nontrivial solution, since we do not have the invariance by translations of the energy functional in this case.

\begin{proposition}
	The weak limit $(u_{0},v_{0})$ is nontrivial.
\end{proposition}
\begin{proof}
	We suppose by contradiction that $(u_{0},v_{0})=(0,0)$. Thus, we have
	\begin{itemize}
		\item $u_{n}\rightarrow u_{0}$ strongly in $L^{r}_{loc}(\mathbb{R}^{N})$, for all $p\leq r<p^{*}$;
		\item $v_{n}\rightarrow v_{0}$ strongly in $L^{s}_{loc}(\mathbb{R}^{N})$, for all $q\leq s<q^{*}$;
		\item $u_{n}(x)\rightarrow u_{0}(x)$ and $v_{n}(x)\rightarrow v_{0}(x)$, almost everywhere in $\mathbb{R}^{N}$.
	\end{itemize}
	It follows by assumption \ref{v4} that for any $\varepsilon>0$ there exists $R>0$ such that
	\begin{equation}\label{ej24}
	|a_{\mathrm{o}}(x)-a(x)|<\varepsilon, \quad |b_{\mathrm{o}}(x)-b(x)|<\varepsilon, \quad |\lambda(x)-\lambda_{\mathrm{o}}(x)|<\varepsilon, \quad \mbox{for all} \hspace{0,2cm} x\in B_{R}(0)^{c}.
	\end{equation}
	Using \eqref{ej24} and the local convergence we deduce that
	\begin{eqnarray}
	\left|\int_{\mathbb{R}^{N}}(a_{\mathrm{o}}(x)-a(x))|u_{n}|^{p}\;\mathrm{d} x\right| & \leq & \int_{B_{R}(0)}|a_{\mathrm{o}}(x)-a(x)||u_{n}|^{p}\;\mathrm{d} x+C\varepsilon\int_{B_{R}(0)^{c}}|u_{n}|^{p}\;\mathrm{d} x\nonumber\\
	& \leq & (\|a_{\mathrm{o}}\|_{\infty}+\|a\|_{\infty})\varepsilon+C\varepsilon,\label{ej25}
	\end{eqnarray}	
	for all $n\geq n_{0}$. Analogously we get
	\begin{equation}\label{ej26}
	\left|\int_{\mathbb{R}^{N}}(b_{\mathrm{o}}(x)-b(x))|v_{n}|^{q}\;\mathrm{d} x\right|\leq (\|b_{\mathrm{o}}\|_{\infty}+\|b\|_{\infty})\varepsilon+C\varepsilon.
	\end{equation}
	Moreover, using H\"{o}lder inequality with $\alpha/p+\beta/q=1$ we deduce that
	\begin{equation}\label{ej28}
	\left|\int_{\mathbb{R}^{N}}(\lambda(x)-\lambda_{\mathrm{o}}(x))|u_{n}|^{\alpha}|v_{n}|^{\beta}\;\mathrm{d} x\right|\leq (\|\lambda\|_{\infty}+\|\lambda_{\mathrm{o}}\|_{\infty})\varepsilon+C\varepsilon.
	\end{equation}
	Combining \eqref{ej25}, \eqref{ej26} and \eqref{ej28} we conclude that
	\[
	I_{\mathrm{o}}(u_{n},v_{n})-I(u_{n},v_{n})=o_{n}(1) \quad \mbox{and} \quad \left\langle I_{\mathrm{o}}'(u_{n},v_{n})-I'(u_{n},v_{n}),\left(\frac{1}{p}u_{n},\frac{1}{q}v_{n}\right)\right\rangle=o_{n}(1),
	\]
	which jointly with \eqref{ej21} implies that
	\begin{equation}\label{ej29}
	I_{\mathrm{o}}(u_{n},v_{n})= c_{\mathcal{N}}+o_{n}(1) \quad \mbox{and} \quad \left\langle I_{\mathrm{o}}'(u_{n},v_{n}),\left(\frac{1}{p}u_{n},\frac{1}{q}v_{n}\right)\right\rangle=o_{n}(1).
	\end{equation}
	In light of Lemma~\ref{nehari-1} we get a sequence $(t_{n})_{n}\subset(0,+\infty)$ such that $(t_{n}^{1/p}u_{n},t_{n}^{1/q}v_{n})_{n}\subset\mathcal{N}_{\mathrm{o}}$.
	
	\vspace{0,3cm}
	
	\noindent \textit{Claim 1.} $\limsup_{n\rightarrow+\infty}t_{n}\leq 1$.
	
	\vspace{0,3cm}
	
	We suppose by contradiction that the claim does not hold, that is, there exists $\varepsilon_{0}>0$ such that, up to a subsequence, we have $t_{n}\geq 1+\varepsilon_{0}$, for all $n\in\mathbb{N}$. By using \eqref{ej29} and the fact that $(t_{n}^{1/p}u_{n},t_{n}^{1/q}v_{n})_{n}\subset\mathcal{N}_{\mathrm{o}}$ we obtain
	\[
	\frac{1}{p}\int_{\mathbb{R}^{N}}\left(\frac{f(t_{n}^{1/p}u_{n})}{t_{n}^{1-\frac{1}{p}}}u_{n}-f(u_{n})u_{n}\right)\;\mathrm{d} x+\frac{1}{q}\int_{\mathbb{R}^{N}}\left(\frac{g(t_{n}^{1/q}v_{n})}{t_{n}^{1-\frac{1}{q}}}v_{n}-g(v_{n})v_{n}\right)\;\mathrm{d} x=o_{n}(1).
	\]
	Since $t_{n}\geq 1+\varepsilon_{0}$, it follows from \eqref{ej30} and \eqref{ej31} that
	\[
	\frac{1}{p}\int_{\mathbb{R}^{N}}\left(\frac{f((1+\varepsilon_{0})^{1/p}u_{n})}{(1+\varepsilon_{0})^{1-\frac{1}{p}}}u_{n}-f(u_{n})u_{n}\right)+\frac{1}{q}\int_{\mathbb{R}^{N}}\left(\frac{g((1+\varepsilon_{0})^{1/q}v_{n})}{(1+\varepsilon_{0})^{1-\frac{1}{q}}}v_{n}-g(v_{n})v_{n}\right)\leq o_{n}(1).
	\]
	Arguing as in Proposition~\ref{p6}, we introduce the sequence $(\tilde{u}_{n}(x),\tilde{v}_{n}(x))=(u_{n}(x+y_{n}),v_{n}(x+y_{n}))$, which is bounded in $E$ and, up to a subsequence, $(\tilde{u}_{n},\tilde{v}_{n})\rightharpoonup (\tilde{u}_{0},\tilde{v}_{0})$ weakly in $E$. Moreover, $(\tilde{u}_{0},\tilde{v}_{0})\neq(0,0)$. Thus, using \eqref{ej30}, \eqref{ej31} and Fatou's Lemma we get
	\[
	0<\frac{1}{p}\int_{\mathbb{R}^{N}}\left(\frac{f((1+\varepsilon_{0})^{1/p}u_{0})}{(1+\varepsilon_{0})^{1-\frac{1}{p}}}u_{0}-f(u_{0})u_{0}\right)+\frac{1}{q}\int_{\mathbb{R}^{N}}\left(\frac{g((1+\varepsilon_{0})^{1/q}v_{0})}{(1+\varepsilon_{0})^{1-\frac{1}{q}}}v_{0}-g(v_{0})v_{0}\right)\leq o_{n}(1),
	\]
	which is not possible and finishes the proof of \textit{Claim 1}.
	
	\vspace{0,3cm}
	
	\noindent \textit{Claim 2.} There exists $n_{0}\in\mathbb{N}$ such that $t_{n}\geq1$, for all $n\geq n_{0}$.
	
	\vspace{0,3cm}
	
	We suppose by contradiction that $t_{n}<1$ for all $n\in\mathbb{N}$. Thus, $t_{n}^{1/p}\leq t_{n}^{1/q}\leq 1$. Hence, using Lemma~\ref{gs} and the fact that $(t_{n}^{1/p}u_{n},t_{n}^{1/q}v_{n})_{n}\subset\mathcal{N}_{\mathrm{o}}$ we obtain
	\begin{eqnarray*}
		c_{\mathcal{N}_{\mathrm{o}}} & \leq & \frac{1}{p}\int_{\mathbb{R}^{N}}(f(t_{n}^{1/p}u_{n})t_{n}^{1/p}u_{n}-pF(t_{n}^{1/p}u_{n}))\;\mathrm{d} x+\frac{1}{q}\int_{\mathbb{R}^{N}}(g(t_{n}^{1/q}v_{n})t_{n}^{1/q}v_{n}-qG(t_{n}^{1/q}v_{n}))\;\mathrm{d} x\\
		& \leq & \frac{1}{p}\int_{\mathbb{R}^{N}}(f(u_{n})u_{n}-pF(u_{n}))\;\mathrm{d} x+\frac{1}{q}\int_{\mathbb{R}^{N}}(g(v_{n})v_{n}-qG(v_{n}))\;\mathrm{d} x\\
		& = & c_{\mathcal{N}}+o_{n}(1),
	\end{eqnarray*}
	which implies that $c_{\mathcal{N}_{\mathrm{o}}}\leq c_{\mathcal{N}}$ and contradicts Lemma~\ref{rel}.
	
	By using \textit{Claims} $1$ and $2$ we can deduce that
	\begin{equation}\label{ej32}
	\int_{\mathbb{R}^{N}}(F(t_{n}^{1/p}u_{n})-F(u_{n}))\;\mathrm{d} x=\int_{1}^{t_{n}^{1/p}}\int_{\mathbb{R}^{N}}f(\tau u_{n})u_{n}\;\mathrm{d} x=o_{n}(1),
	\end{equation}
	\begin{equation}\label{ej33}
	\int_{\mathbb{R}^{N}}(G(t_{n}^{1/q}v_{n})-G(v_{n}))\;\mathrm{d} x=\int_{1}^{t_{n}^{1/q}}\int_{\mathbb{R}^{N}}g(\tau v_{n})v_{n}\;\mathrm{d} x=o_{n}(1).
	\end{equation}
	Moreover, since $a_{\mathrm{o}},b_{\mathrm{o}}\in L^{\infty}(\mathbb{R}^{N})$ and $(u_{n},v_{n})_{n}$ is bounded in $E_{\mathrm{o}}$ we also have
	\begin{equation}\label{ej34}
	(t_{n}-1)\left(\frac{1}{p}\|u_{n}\|_{a_{\mathrm{o}},p}^{p}+\frac{1}{q}\|v_{n}\|_{b_{\mathrm{o}},q}^{q}- \int_{\mathbb{R}^{N}}\lambda_{\mathrm{o}}(x)|u_{n}|^{\alpha}|v_{n}|^{\beta}\;\mathrm{d} x\right)=o_{n}(1).
	\end{equation}
	Combining \eqref{ej32}, \eqref{ej33} and \eqref{ej34} we conclude that
	\[
	I_{\mathrm{o}}(t_{n}^{1/p}u_{n},t_{n}^{1/q}v_{n})-I_{\mathrm{o}}(u_{n},v_{n})=o_{n}(1).
	\]
	Thus, in view of \eqref{ej29} we get
	\[
	c_{\mathcal{N}_{\mathrm{o}}}\leq I_{\mathrm{o}}(t_{n}^{1/p}u_{n},t_{n}^{1/q}v_{n})=I_{\mathrm{o}}(u_{n},v_{n})+o_{n}(1)=c_{\mathcal{N}}+o_{n}(1),
	\]
	which contradicts Lemma~\ref{rel}. Therefore, $(u_{0},v_{0})\neq(0,0)$.
\end{proof}

\begin{proof}[Proof of Theorem~\ref{B} completed.]
	Since $(u_{0},v_{0})$ is a nontrivial critical point of $I$, we have that $(u_{0},v_{0})\in\mathcal{N}$. Hence, $c_{\mathcal{N}}\leq I(u_{0},v_{0})$. On the other hand, it follows from \eqref{ej21} and Fatou's Lemma that
	\begin{eqnarray*}
		c_{\mathcal{N}}+o_{n}(1) & = & \frac{1}{p}\int_{\mathbb{R}^{N}}(f(u_{n})u_{n}-pF(u_{n}))\;\mathrm{d} x+\frac{1}{q}\int_{\mathbb{R}^{N}}(g(v_{n})v_{n}-qG(v_{n}))\;\mathrm{d} x\\
		& \geq & \frac{1}{p}\int_{\mathbb{R}^{N}}(f(u_{0})u_{0}-pF(u_{0}))\;\mathrm{d} x+\frac{1}{q}\int_{\mathbb{R}^{N}}(g(v_{0})v_{0}-qG(v_{0}))\;\mathrm{d} x+o_{n}(1)\\
		& = & I(u_{0},v_{0})+o_{n}(1),
	\end{eqnarray*}
	which implies that $c_{\mathcal{N}}\geq I(u_{0},v_{0})$. Therefore, $(u_{0},v_{0})$ is a ground state for System~\eqref{ej00}. By a similar argument used in Propositions~\ref{p4}, \ref{p7} and \ref{p5}, we obtain $t_{0}>0$ such that $(t_{0}^{1/p}|u_{0}|,t_{0}^{1/q}|v_{0}|)\in\mathcal{N}$ is a positive ground state solution for System~\eqref{ej00}, for some $\lambda>0$.
\end{proof}


\begin{acknowledgement}
	Research supported in part by INCTmat/MCT/Brazil, CNPq and CAPES/Brazil. The second author was also partially supported by Fapego Fapeg/CNpq grants 03/2015-PPP.
\end{acknowledgement}


\bigskip

\end{document}